\documentclass[11pt]{amsart}

\usepackage{andy}
\usepackage{ytableau}
\usepackage{framed}

\usepackage{amsmath}
\usepackage{amssymb}

\usetikzlibrary{calc,decorations.pathmorphing,shapes,snakes,arrows}

\newcommand{\dfn}[1]{\textit{#1}}
\newcommand{\U}{\cal U}

\newcommand{\low}{\cal L}

\newcommand{\E}{E}

\newcommand{\UP}{\ensuremath{U_{\hspace{-0.6mm}P}}}

\newcommand{\anti}[1]{\cal I^{#1}}
\newcommand{\chain}[1]{\cal C^{#1}}

\newcommand{\D}[0]{\cal D}
\newcommand{\Y}[0]{\ensuremath{\mathrm{\mathbf{Y}}}}

\renewcommand{\preceq}{\preccurlyeq}

\DeclareMathOperator{\Ad}{Ad}
\DeclareMathOperator{\supp}{supp}
\DeclareMathOperator{\rel}{rel}
\DeclareMathOperator{\stab}{stab}
\DeclareMathOperator{\ac}{ac}
\DeclareMathOperator{\lb}{lb}
\DeclareMathOperator{\ub}{ub}

\newcounter{sarrow}
\newcommand\xrsquigarrow[1]{%
    \stepcounter{sarrow}%
    \mathrel{\begin{tikzpicture}[baseline= {( $ (current bounding box.south) + (0,-0.5ex) $ )}]
        \node[inner sep=.5ex] (\thesarrow) {$\scriptstyle #1$};
        \path[draw,<-,decorate,
        decoration={zigzag,amplitude=0.7pt,segment length=1.2mm,pre=lineto,pre length=4pt}]
        (\thesarrow.south east) -- (\thesarrow.south west);
    \end{tikzpicture}}%
}

\newcommand\xrrsquigarrow[1]{%
    \stepcounter{sarrow}%
    \mathrel{\begin{tikzpicture}[baseline= {( $ (current bounding box.south) + (0,-0.5ex) $ )}]
        \node[inner sep=.5ex] (\thesarrow) {$\scriptstyle #1$\hspace{.5ex} };
        \path[draw,double,<-,decorate,
        decoration={zigzag,amplitude=0.7pt,segment length=1.2mm,pre=lineto,pre length=4pt}]
        (\thesarrow.south east) -- (\thesarrow.south west);
    \end{tikzpicture}}%
}

\let\wkemb=\rightsquigarrow
\newcommand{\wkembin}[1]{ \xrsquigarrow{#1} }

\newcommand{\embeds}[0]{\xrrsquigarrow{\hspace{.5ex}} }
\newcommand{\embedsin}[1]{ \xrrsquigarrow{#1} }

\def\emp{\nothing}

\def\zz{\mathbb Z}
\def\nn{\mathbb N}

\def\rr{\mathbb R}

\def\fq{{\mathbb F}_q}

\def\si{\sigma}

\def\<{\langle}
\def\>{\rangle}

\def\width{\text{{\rm width}}}

\def\U{{\mathcal{U} } }

\def\0{{\mathbf 0}}

\def\nothing{\varnothing}

\def\.{\hskip.06cm}
\def\ts{\hskip.03cm}

\def\bb{{\textbf{\textit{B}}}\hspace{-.01cm}}

\def\dz{z}
\def\ra{ {\text {\rm A}  } }
\def\rb{ {\text {\rm B}  } }

\title{On Higman's $k(U_n(q))$ Conjecture}

\author[Igor~Pak]{ \ Igor~Pak$^\star$ \,}

\author[ Andrew Soffer]{ \ Andrew Soffer$^\star$}

\thanks{\thinspace~${\hspace{-.45ex}}^\star $Department of Mathematics,
UCLA, Los Angeles, CA, 90095.
\hskip.06cm
Email:
\hskip.06cm
\texttt{\{pak,\ts\/asoffer\}@math.ucla.edu}}

\date{\today}

\begin{document}

\begin{abstract}
  A classical conjecture by Graham Higman states that the number of conjugacy classes of $U_n(q)$, the group of upper triangular $n\times n$ matrices over~$\fq$, is polynomial in~$q$, for all~$n$.
  In this paper we present both positive and negative evidence, verifying the conjecture for $n\le 16$, and suggesting that it probably fails for $n\ge 59$.

  The tools are both theoretical and computational.
  We introduce a new framework for testing Higman's conjecture, which involves recurrence relations for the number of conjugacy classed of \emph{pattern groups}.
  These relations are proved by the \emph{orbit method} for finite nilpotent groups.
  Other applications are also discussed.
\end{abstract}

\maketitle

\section{Introduction}


Let $k(G)$ denote the number of conjugacy classes of a finite group~$G$, and let $U_n(q)$ be the group of upper triangular $n \times n$ matrices over a finite field $\fq$ with ones on the diagonal.
In~\cite{H1}, Higman made the following celebrated conjecture:

\begin{conj}[Higman, 1960]\label{conj:higman}
  For every positive integer $n$, the number of conjugacy classes in $U_n(q)$ is a polynomial
  function of~$q$.
\end{conj}

Higman was motivated by the problem of enumerating $p$-groups.
Since then, much effort has been made to verify and establish the result.
Notably, Arregi and Vera-L\'opez verified Higman's conjecture for $n\le 13$ in~\cite{VA3} (see below and \S\ref{fin_rems_va}).
More recently, John Thompson~\cite{Tho} laid some ground towards a positive resolution of the conjecture, but the proof remains elusive.

\smallskip

In this paper we make a new push towards resolving the conjecture,
presenting both positive and negative evidence.
Perhaps surprisingly, results of both type are united by the same
underlying idea of embedding smaller pattern groups into larger ones (see below).

\begin{thm}\label{thm:n16}
  Higman's conjecture holds for all \ts $n\le 16$.  Moreover, for all \ts $n \le 16$, we have
  $k(U_n) \in \nn[q-1]$.
\end{thm}

This extends the results of Arregi and Vera-L\'opez and earlier computational results
in favor of Higman's conjecture.  Our approach is based on computing the polynomials
indirectly via a recursion over certain co-adjoint orbits arising in the finite
field analogue of Kirillov's \emph{orbit method} (see~\cite{K2,K3}).
This approach is substantially different and turns out to be significantly more efficient than the previous
work which is based on direct enumeration of the conjugacy classes.
We present the algorithm proving Theorem~\ref{thm:n16} in Section~\ref{sec:experimental}
and describe the earlier work in Section~\ref{sec:fin-rems}.

Our approach is based on a recursion over a large class of pattern groups. A \emph{pattern group} is
a subgroup of $U_n(q)$ where some matrix entries are fixed to be zeroes.
In a recent paper~\cite{HP}, Halasi and P\'alfy showed that the analogue of
Higman's conjecture fails for certain pattern groups. In fact, they show that
$k(\UP(q))$ can be as bad as one desires, e.g.~non-polynomial even when the characteristic of $\fq$ is fixed.
This work was the starting point of our investigation. Our next two result are also computational.

\begin{thm}\label{thm:hp-small}
  For every pattern subgroup $\UP(q) \leqslant U_{9}(q)$, we have $k(\UP(q))\in\nn[q-1]$.
\end{thm}

While this shows that small pattern groups do exhibit polynomial behavior,
this is is false for larger~$n$.

\begin{thm}\label{thm:hp-new}
  There is a pattern subgroup $\UP(q) \leqslant U_{13}(q)$ such that $k(\UP(q))$ is \emph{not} a polynomial function of~$q$.
\end{thm}

Note that while Halasi and P\'alfy's approach is constructive,
they do not give an explicit bound on the size of such a pattern group
(c.f.~\S\ref{ssec:fin-rems-posets}).  We believe that the constant 13 in
Theorem~\ref{thm:hp-new} is optimal, but this computation remains out of
reach in part due to the excessively large number of pattern groups to
consider (see Section~\ref{sec:experimental}).

Our final result offers an evidence against Higman's conjecture:

\begin{thm}\label{thm:u59}
  The pattern subgroup $\UP(q)$ from Theorem~\ref{thm:hp-new} \emph{embeds} into $U_{59}(q)$.
\end{thm}

Here the notion of \emph{embedding} is somewhat technical and iterative.
In Section~\ref{sec:embedding}, we prove that
\[k\left(U_{n}(q)\right) = \sum_{P} F_{P,n}(q) \cdot k\left(\UP(q)\right),\]
where $F_{P,n}(q) \in \zz[q]$ are polynomials and the sum is over pattern subgroups
$\UP(q)$ which embed into $U_{n}(q)$, and are irreducible in a certain formal sense.
Taking Theorem~\ref{thm:hp-new} into account, this strongly suggests
that $k\left(U_{n}(q)\right)$ is not polynomial for sufficiently large~$n$.

\begin{conj}\label{conj:false-59}
  The number of conjugacy classes $k\left(U_{n}(q)\right)$ is \emph{not} polynomial for $n \ge 59$.
\end{conj}

Of course, this conjecture is hopelessly beyond the means of a computer experiment.
We believe that Theorem~\ref{thm:n16} can in principle be extended to $n\le 18$
by building upon our approach, and parallelizing the computation
(see~\S\ref{ssec:fin-rems-parallel}).  It is unlikely however, that this
would lead to a disproof of Higman's Conjecture~\ref{conj:higman}
without a new approach.

Curiously, this brings the status of Higman's conjecture in line with that of Higman's related but more famous \emph{PORC conjecture} (see~\cite{BNV}).  It was stated in 1960, also by Graham Higman, in a followup paper~\cite{H2}.
Higman conjectured that the number $f(p^n)$ of $p$-groups of order $p^n$
is a polynomial on each fixed residue class modulo some~$m$. Higman showed
that the number $f(p^n)$ can also be expressed as a large sum over certain
\emph{descendants}.
%
%
Recently, Vaughan-Lee and du Sautoy showed recently
that some of the terms counting the numbers of descendants is non-polynomial~\cite{DVL}.
Here is how Vaughan-Lee eloquently explains this in~\cite{VL}:

\smallskip
\begin{quotation}
  {
    ``The grand total might still be PORC, even though we know that one of the individual
    summands is not PORC. My own
    view is that this is extremely unlikely. But in any case I believe that Marcus's group
    provides a counterexample to what I hazard to call the \emph{philosophy} behind Higman's
  conjecture.''}
\end{quotation}
\smallskip

\noindent
We are hoping the reader views our results in a similar vein
(c.f.~\S\ref{ssec:fin-rems-alperin}).
%
%

\smallskip

The rest of this paper is structured as follows.  We begin with
definitions and notation in Section~\ref{sec:defs}.  In
Section~\ref{sec:pattern}, we prove some preliminary results
on co-adjoint orbits of the pattern groups.  We then proceed
to develop combinatorial tools giving recursions for the number
of co-adjoint orbits (Section~\ref{sec:combo_tools}).
Section~\ref{sec:embedding} is essentially poset theoretic,
which allows us to combine the results to prove theorems~\ref{thm:hp-new}
and~\ref{thm:u59}.  The experimental work which proves theorems~\ref{thm:n16}
and~\ref{thm:hp-small} is given the Section~\ref{sec:experimental}.
We conclude with final remarks and open problems in Section~\ref{sec:fin-rems}.

\bigskip

\section{Definitions and notation}\label{sec:defs}
For any finite group $G$, we write $k(G)$ for the number of conjugacy classes in $G$.
We assume the notation that $q=p^r$ is a prime power, and we denote by $\fq$ the finite field with $q$ elements.
In the matrix ring $M_{n\times n}(\fq)$, the element $e_{i,j}$ will denote the element which is one in cell~$(i,j)$ and zero everywhere else.
For $\alpha\in\fq^\times$, we will let $\E_{i,j}(\alpha)$ denote the elementary transvection $\E_{i,j}(\alpha):=1+\alpha e_{i,j}$.

Throughout the paper, all posets are finite, and typically denoted
by the letters $P$ and~$Q$.  We adopt the following notation regarding posets (c.f.~\cite{Sta}).
By a slight abuse of notation, we identify poset $P$ with a ground set on which
partial order~``$\prec$" is defined.
We use $\chain n$ and $\anti n$ to denote the $n$-chain and $n$-antichain,
respectively.  We use $\max(P)$ and $\min(P)$ to denote the set of maximal
and minimal elements, respectively.  The set of anti-chains in $P$ is denoted by~$\ac(P)$.
The set of pairs of distinct related elements on a poset $P$ is denoted
\[\rel(P):=\{(x,y): x\prec_P y\}.\]
When the poset $P$ is clear from context, we omit the subscript on the relation.
The \dfn{upper} and \dfn{lower bounds} of an element $x\in P$ are defined as
\[\ub_P(x):=\{y\in P: x\prec y\} \quad \text{and} \quad \lb_P(x):=\{y\in P: y\prec x\}.
\]

For a subset $S \subseteq P$, let $P|_S$ denote the subposet of $P$ induced on the set $S$.
As a special case, for $x\in P$, we write $P-x$ for the subposet of $P$ induced on $P\setminus\{x\}$.
For an element $x\in P$, let $P^{(x)}$ to be the poset consisting exclusively of
the relations where the larger element is $x$. That is, we have:
\[\rel\left( P^{(x)} \right)=\{(w,x): w\prec_P x\}.\]
We say that a poset $P$ is $Q$-free if no induced subposet of $P$ is isomorphic to $Q$.
For example, a poset is $\anti 2$-free if and only if it is a chain.
Similarly, a poset is $\chain2$-free if and only if it is $\anti n$.

For a poset $P$, the \dfn{dual} poset $P^*$ will be the one whose relations are reversed.
That is, if $x\prec_P y$, then $y\prec_{P^*}x$.
We also define two constructions of posets from smaller ones.
First, for posets $P$ and $Q$, their \dfn{disjoint union}
$P\amalg Q$ is a poset whose elements are the elements of $P$ and~$Q$,
and for which $x\prec y$ if either
\begin{enumerate}
  \item $x,y\in P$, and $x\prec_P y$, or
  \item $x,y\in Q$, and $x\prec_Q y$.
\end{enumerate}
Clearly, up to isomorphism, the operation $\amalg$ is both commutative and associative.
Second, the \dfn{lexicographic sum} $P+Q$ is the poset whose elements are the elements of $P$ and $Q$, and for which $x\prec y$ if any of the following hold:
\begin{enumerate}
  \item $x,y\in P$, and $x\prec_P y$,
  \item $x,y\in Q$, and $x\prec_Q y$, or
  \item $x\in P$ and $y\in Q$.
\end{enumerate}
In terms of the Hasse diagrams (the usual graphical representation of a poset), the lexicographic sum is obtained by placing $Q$ above $P$.
Hence, it is clear that the lexicographic sum is not commutative, but is associative (up to isomorphism).

\bigskip

\section{Pattern groups and the co-adjoint action}\label{sec:pattern}

\subsection{Definitions and basic results}
We recall the definition of pattern algebras and pattern groups given by Isaacs in~\cite{Isa}, but with slightly different notation.
Let $P$ be a poset on $\{1,2,\dots,n\}$ which has the standard ordering as a linear extension.
That is, whenever~$i\preceq_P j$, then we also have~$i\le j$.
Define the \dfn{pattern algebra} $\U_P(q)$ to be \[\U_P(q):=\{X\in M_{n\times n}(\fq) : X_{i,j}=0\text{ if }i\not\prec j\}.\]
Every pattern algebra $\U_P(q)$ is a nilpotent $\fq$-algebra.
In fact, the pattern algebra~$\U_P(q)$ is a subalgebra of the strictly upper-triangular matrices $\U_n(q)$.\footnote{Note that $\U_P$ depends not only on the isomorphism class of the poset $P$, but also on a specific linear extension of $P$.
  This definition is purely one of convenience.
  One could define $\U_P$ abstractly in terms of generators and relations in such a way as to make it clear that if $P$ and $Q$ are isomorphic posets, then $\U_P\cong\U_Q$.
We use this isomorphism throughout the paper without further mention.}

Define the \dfn{pattern group} $\UP(q):=\{1+X : X\in\U_P(q)\}$.
For general posets $P$, the group $\UP(q)$ is a subgroup of the unitriangular group $U_n(q)$.
To simplify notation, we often omit the field and write $\U_P$ instead of $\U_P(q)$. Similarly, we abbreviate $\UP(q)$ by $\UP$.

\begin{exmp}
  We exhibit several examples of posets and their associated pattern algebras and pattern groups.
  \begin{enumerate}
    \item If $P=\chain n$, then $\U_P=\U_n$, and $\UP=U_n$.
    \item If $P=\anti n$, then $\U_P$ is the trivial algebra, and $\UP$ consists only of the identity matrix.
    \item If $P$ is the poset in Figure~\ref{fig:pattern_alg_exmp}, then $\UP$ consists of matrices of the form shown in the figure.
      We can see that as a vector space, $\U_P$ is generated by \[\{e_{1,2},\ e_{1,3},\ e_{1,4},\ e_{1,5},\ e_{2,3},\ e_{2,4},\ e_{2,5},\ e_{3,4}\}.\]
      These are precisely the elements $e_{i,j}$ where $i\prec_P j$.
      As an algebra, $\U_P$ can be generated by fewer elements.
      In particular, the pattern algebra $\U_P$ can be generated (as an algebra) by $\{e_{1,2},e_{2,3},e_{3,4},e_{2,5}\}$.
      Note that $e_{i,j}$ is in this set precisely when $i$ and $j$ are connected by a line segment in the Hasse diagram (see Figure~\ref{fig:pattern_alg_exmp}).
      The generators are the minimal relations (in the language of posets, the \dfn{cover relations}).
  \end{enumerate}

  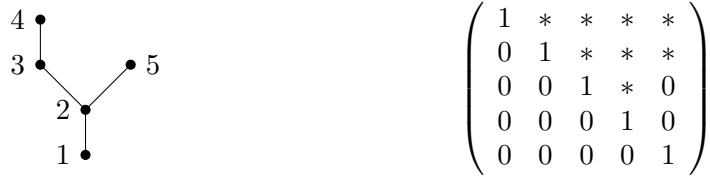
\begin{figure}[h]
    \begin{minipage}{0.45\linewidth}
      \centering
      \begin{tikzpicture}[scale=0.6]
        \draw[fill] (1,2) circle (.1cm);
        \node at (1.5,2) {5};

        \draw[fill] (-1,3) circle (.1cm);
        \node at (-1.5,3) {4};

        \draw[fill] (-1,2) circle (.1cm);
        \node at (-1.5,2) {3};

        \draw[fill] (0,1) circle (.1cm);
        \node at (-0.5,1) {2};

        \draw[fill] (0,0) circle (.1cm);
        \node at (-0.5,0) {1};

        \draw (1,2) -- (0,1);
        \draw (-1,3) -- (-1,2);
        \draw (-1,2) -- (0,1);
        \draw (0,1) -- (0,0);
      \end{tikzpicture}

    \end{minipage}
    \begin{minipage}{0.45\linewidth}
      \[\left( \begin{array}{ccccc}
          1 & \ast & \ast & \ast & \ast\\
          0 & 1 & \ast & \ast & \ast\\
          0 & 0 & 1 & \ast & 0\\
          0 & 0 & 0 & 1 & 0\\
          0 & 0 & 0 & 0 & 1
      \end{array} \right)\]
    \end{minipage}
    \caption{A poset $P$ and the form of elements in the associated pattern group $\UP$.}
    \label{fig:pattern_alg_exmp}
  \end{figure}
\end{exmp}

Pattern groups have a particularly nice presentation which we will need in Section~\ref{sec:combo_tools}.

\begin{prop}\label{prop:generators}
  For every poset $P$, we have
\[\UP(q)\. =\. \angs{\E_{i,j}(\alpha)\middle|i\prec_P j,\alpha\in\fq^\times}.\]
  Moreover, for every $\alpha,\beta\in\fq^\times$, we have
\[[\E_{i,j}(\alpha),\E_{k,\ell}(\beta)]\. = \.
\begin{cases}
    \E_{i,\ell}(\alpha\beta) & \text{if }j=k\ts,\\
    1 & \text{if }j\ne k\ts.
\end{cases}\]
\end{prop}

\medskip

\subsection{The adjoint and co-adjoint actions for pattern groups}
The \dfn{adjoint action} of $\UP$ on $\U_P$ is defined by
\[\Ad_g:X\mapsto gXg^{-1}\] for $g\in \UP$ and $X\in\U_P$.
Enumerating conjugacy classes of a pattern group is equivalent to enumerating orbits of the adjoint action.
Indeed, the action of $\UP$ on itself by conjugation is equivariant with the adjoint action, as \[1+\Ad_g(X)=1+gXg^{-1}=g(1+X)g^{-1}.\]

We consider the \dfn{co-adjoint action} of $\UP$ on the dual of $\U_P$.
For $f\in\U_P^*$ and $g\in \UP$, define \ts $K_g(f)\in\U_P^*$ \ts by
\[K_g(f):X\mapsto f(g^{-1}Xg),\]
for all $X\in\U_P$.
In other words, the co-adjoint action is given by \ts $K_g(f)=f\circ\Ad_{g^{-1}}$.

\begin{lem}\label{lem:adcoad}
  The number of co-adjoint orbits for a pattern group $\UP(q)$
  is equal to the number of adjoint orbits, and hence $k(\UP(q))$.
\end{lem}

Note that several versions of the lemma are known (c.f~\cite{K1,K2}).
In particular, Kirillov proves the special case of $U_n$ in~\cite{K3}.
We present a full proof here for completeness.

\begin{proof}
Extend both the adjoint and the co-adjoint actions by linearity
from $\UP$ to the entire group algebra $\zz[\UP]$.
Then for $f\in\U_P^*$, the co-adjoint action $K_{g-1}$ annihilates~$f$
if and only if~$f$ vanishes on the image of $\Ad_{g^{-1}-1}$.
Indeed, \[K_{g-1}(f)=K_g(f)-f=f\circ\Ad_{g^{-1}}-f=f\circ\Ad_{g^{-1}-1}.\]
Let $I_g=\im(\Ad_{g^{-1}-1})$.
We apply Burnside's lemma to count the orbits of the co-adjoint action:
  \begin{align*}
    \bigl|\U_P^*/\UP\bigr| \, &= \, \frac1{\abs{\UP}}\. \sum_{g\in \UP}\. \bigl|\ker(K_{g-1})\bigr| \, = \,
    \frac1{\abs{\UP}}\. \sum_{g\in \UP} \. \#\left\{f\in\U_P^*\. | \. I_g\subseteq\ker f\right\}\\
    &= \, \frac1{\abs{\UP}}\. \sum_{g\in \UP}\. q^{\dim\U_P-\dim I_g} \, = \,
    \frac1{\abs{\UP}}\. \sum_{g\in \UP}\. \bigl|\ker(\Ad_{g^{-1}-1})\bigr| \, = \, \bigl|\U_P/\UP\bigr|\ts.
  \end{align*}
  This completes the proof.
\end{proof}

\subsection{An explicit realization of the co-adjoint action}\label{subsec:lower}
In place of functionals on pattern algebras, we identify $\U^*_P$ with a quotient space of matrices.
Define
\[\low_P(q):=\left.M_{n\times n}(\fq)\middle/\bigoplus_{i\not\succ j}\fq e_{i,j}\right.\..\]

When $P=\chain n$ (the total order $\{1<\dots<n\}$), then $\low_P$ is the space of lower triangular matrices thought of as a quotient of all matrices by upper-triangular matrices (hence the notation ``$\low$'').
For general posets $P$, the space $\low_P$ is a quotient spaces of lower triangular matrices.

\begin{figure}[h]
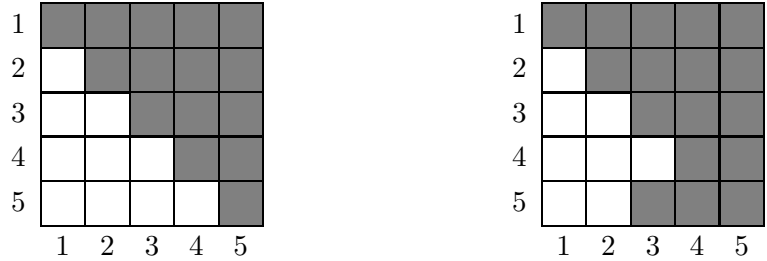

  \centering
  \begin{minipage}{0.45\textwidth}
    \centering
    \begin{ytableau}
      \none[1] & *(gray) & *(gray) & *(gray) & *(gray) & *(gray)\\
      \none[2] & & *(gray) & *(gray) & *(gray) & *(gray)\\
      \none[3] & & & *(gray) & *(gray) & *(gray)\\
      \none[4] & & & & *(gray) & *(gray)\\
      \none[5] & & & & & *(gray) \\
      \none & \none[1] & \none[2] & \none[3] & \none[4] & \none[5]
    \end{ytableau}
  \end{minipage}
  \begin{minipage}{0.45\textwidth}
    \centering
    \begin{ytableau}
      \none[1] & *(gray) & *(gray) & *(gray) & *(gray) & *(gray)\\
      \none[2] & & *(gray) & *(gray) & *(gray) & *(gray)\\
      \none[3] & & & *(gray) & *(gray) & *(gray)\\
      \none[4] & & & & *(gray) & *(gray)\\
      \none[5] & & & *(gray) & *(gray) & *(gray) \\
      \none & \none[1] & \none[2] & \none[3] & \none[4] & \none[5]
    \end{ytableau}
  \end{minipage}

  \caption{The spaces $\low_{\chain5}$ (left) and $\low_P$ (right) where $P$ is the poset shown in Figure~\ref{fig:pattern_alg_exmp}.
  The shaded cells represent those components of $M_{n\times n}$ which are quotiented away in $\low_P$.}
  \label{fig:low_chain_5}
\end{figure}

The space $\low_P$ is isomorphic to $\U_P^*(q)$.
Specifically, for each $A\in\low_P(q)$, define the functional \ts
$f_X\in\U_P^*$ \ts by
\[
f_X(A)\. := \. \tr(X\cdot A)\ts.
\]
This identification is well-defined, as the quotiented cells in
$\low_P$ (those $e_{i,j}$ with $i\not\succ j$) will precisely align
with the cells that are forced to be zero by the definition of $\UP$.
That is, if $i\not\succ j$, then for $A\in \UP$, we have $A_{j,i}=0$.
Thus, their contribution to the trace will be zero.

Pushing the co-adjoint action through this identification yields an
action of~$\UP$ on $\low_P$, which we also call the co-adjoint action
(and also write~$K_g$).
For $g\in \UP$ and $L\in\low_P$, the action becomes
\[K_g(L)\.=\. gLg^{-1}.\]
To be precise, let \ts $\rho:M_{n\times n}\to\low_P$ \ts
denote the canonical projection map.
For~${X\in\low_P}$, pick a representative $X'\in M_{n\times n}$
so that~$\rho(X')=X$. Then~${K_g(X)=\rho(gX'g^{-1})}$.
It is evident that the choice of such an $X'$ is irrelevant.

\begin{exmp}
  Let $P$ denote the poset shown in Figure~\ref{fig:pattern_alg_exmp}, and let $X\in\low_P(q)$ denote the element shown below on the left.
  We consider the co-adjoint action of the elementary matrix $E=\E_{2,3}(1)$ on $X$.

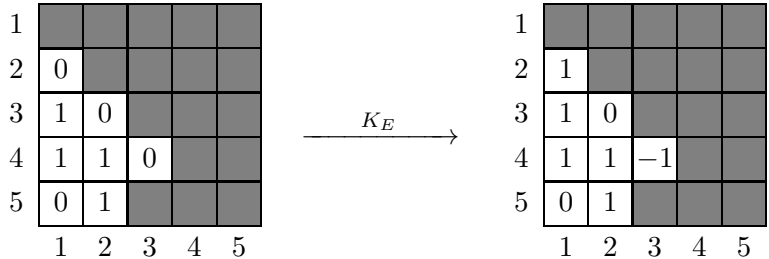
  \begin{center}
    \vspace{3mm}
    \begin{minipage}{0.3\textwidth}\centering
      \begin{ytableau}
        \none[1] & *(gray) & *(gray) & *(gray) & *(gray) & *(gray)\\
        \none[2] & 0 & *(gray) & *(gray) & *(gray) & *(gray)\\
        \none[3] & 1 & 0 & *(gray) & *(gray) & *(gray)\\
        \none[4] & 1 & 1 & 0 & *(gray) & *(gray)\\
        \none[5] & 0 & 1 & *(gray) & *(gray) & *(gray) \\
        \none & \none[1] & \none[2] & \none[3] & \none[4] & \none[5]
      \end{ytableau}
    \end{minipage}
    $\xrightarrow{\hspace{7mm}K_E\hspace{7mm}}$
    \begin{minipage}{0.3\textwidth}\centering
      \begin{ytableau}
        \none[1] & *(gray) & *(gray) & *(gray) & *(gray) & *(gray)\\
        \none[2] & 1 & *(gray) & *(gray) & *(gray) & *(gray)\\
        \none[3] & 1 & 0 & *(gray) & *(gray) & *(gray)\\
        \none[4] & 1 & 1 & -1 & *(gray) & *(gray)\\
        \none[5] & 0 & 1 & *(gray) & *(gray) & *(gray) \\
        \none & \none[1] & \none[2] & \none[3] & \none[4] & \none[5]
      \end{ytableau}
    \end{minipage}
  \end{center}

  Consider the left multiplication $X\mapsto EX$.
  This action adds the contents of row 3 to row 2.
  Thus, for $Y=K_E(X)$, we have $Y_{3,1}=X_{3,1}+X_{3,2}$.
  All other cells in row~2 are trivial in $\low_P$.
  For the right multiplication $EX\mapsto EXE^{-1}$,
  we take the contents of column~2 and subtract them
  from the contents of column~3.  We see that $Y_{4,3}=X_{4,3}-X_{4,2}$.
  All other cells in column~3 are trivial in $\low_P$, so this is the only relevant data.
\end{exmp}

Observe that each of these actions (conjugation, the adjoint action,
and the co-adjoint actions on $\U^*_P$ and $\low_P$) have the same number of orbits.
Therefore,
\[
k(\UP(q)) \. = \. \bigl|\U_P(q)/\UP(q)\bigr|\. = \.
\bigl|\U_P^*(q)/\UP(q)\bigr| \. =\. \bigl|\low_P(q)/\UP(q)\bigl|\ts.
\]
We use the notation $k(P)$ for this quantity.

\bigskip

\section{Combinatorial tools}\label{sec:combo_tools}
In this section we construct several tools to compute $k(P)$
using the structure of the poset~$P$.
We begin with several simple observations which lead to useful tools.

\subsection{Elementary operations}
We begin with a the following result which can be seen easily in the
language of the co-adjoint action on $\low_P$ (see~\S\ref{subsec:lower}).
We prove it via elementary group theory.

\begin{prop}\label{prop:obs}
  For posets $P$ and $Q$, we have
  \begin{enumerate}
    \item $k(P)=k(P^*)$
    \item $k(P)=k(P_1)\cdot k(P_2)$ where $P_i=P|_{S_i}$ for $i=1,2$, and $S_1,S_2\subseteq P$ such that $S_1\cup S_2=P$  and $P|_{S_1\cap S_2}$ contains no relations.
    \item $k(P\amalg Q)=k(P)\cdot k(Q)$
  \end{enumerate}
\end{prop}

\begin{proof}
  For~(1), we must label the elements $P^*$ appropriately so that $i\le j$ whenever $i\preceq_{P^*} j$ (as required by the definition of pattern groups).
  Let $n=\abs{P}$, and for each $i\in P$, relabel the element $i$ with the label $n+1-i$.
  This will reverse the total ordering on $P$ so that it agrees with the partial ordering on $P^*$.
  In terms of matrices, we have expressed $U_{\hspace{-0.6mm}P^*}$ as the elements of $\UP$ ``transposed'' about the anti-diagonal.
  Let~$\phi$ denote this anti-diagonal transposition. Then the map $g\mapsto \phi(g^{-1})$ is an isomorphism between the groups $\UP$ and $U_{\hspace{-0.6mm}P^*}$, proving $k(P)=k(P^*)$.

  For~(2), let $P_i=P|_{S_i}$ for $i=1,2$.  We claim that the following map \ts
  $\psi:U_{\hspace{-0.6mm}P_1}\times U_{\hspace{-0.6mm}P_2}\to \UP$ \ts defined by
  \ts $\psi(g_1,g_2)=g_1g_2$ \ts is the isomorphism. First, note that for
  $g_i\in U_{\hspace{-0.6mm}P_i}$, the elements $g_1$ and $g_2$ commute.
  To this end, it suffices to see that generators commute, which follows
  from the fact that $P|_{S_1\cap S_2}$ has no relations, and Proposition~\ref{prop:generators}.
  Then $\psi$ is a homomorphism, as
  \[\psi(g_1,g_2)\psi(h_1,h_2)\. = \. g_1g_2h_1h_2=g_1h_1g_2h_2\. = \. \psi(g_1h_1,g_2h_2)\ts , \ \text{ and}\]
  \[\psi(g_1,g_2)^{-1}\. = \. g_2^{-1}g_1^{-1}=g_1^{-1}g_2^{-1}\. = \. \psi(g_1^{-1},g_2^{-1})\ts.\]
  Whenever $x\prec_P y$, then either $x\prec_{P_1}y$ or $x\prec_{P_2}y$.
  Therefore, every generator $\E_{x,y}(\alpha)$ of $\UP$ is either a generator of $U_{\hspace{-0.6mm}P_1}$ or $U_{\hspace{-0.6mm}P_2}$, so $\psi$ is surjective.
  It follows from the fact that ${\abs{U_{\hspace{-0.6mm}P_1}\times U_{\hspace{-0.6mm}P_2}}=\abs{\UP}}$ that $\psi$ is an isomorphism, proving that $k(P)=k(P_1)\cdot k(P_2)$.

  For~(3), apply~(2) to the poset $P\amalg Q$ with $S_1=P$ and $S_2=Q$.
\end{proof}

\subsection{Poset systems}
Let $P$ be a subposet of $P'$.
Then $\U_P$ canonically injects into $\U_{P'}$, and so we obtain a canonical projection \[\pi_{P',P}:\low_{P'}\to\low_P.\]
This projection sends $e_{i,j}$ to zero whenever $i\prec_{P'}j$, but $i\not\prec_Pj$.
For specific choices of $P$ and~$P'$, this map can be used effectively to enumerate $k(P')$.

Fix a maximal element $m\in\max(P)$.
Of particular interest will be the poset $P^{(m)}$, defined by \[\rel\left( P^{(m)} \right)=\{(x,m): x\prec_P m\}.\]
To simplify notation, for the remainder of this section, let $Q=P^{(m)}$, and let $\pi=\pi_{P,Q}$.
That is, the projection $\pi$ annihilates all $e_{i,j}\in\low_P$ which are not of the form $e_{m,x}$ for $x\prec_Pm$ (see Figure~\ref{fig:proj_exmp}).

\begin{figure}[h]
  \centering
  \begin{tikzpicture}[scale=0.6]
    \draw[fill] (1,2) circle (.1cm);
    \node at (1.5,2) {5};

    \draw[fill] (-1,3) circle (.1cm);
    \node at (-1.5,3) {4};

    \draw[fill] (-1,2) circle (.1cm);
    \node at (-1.5,2) {3};

    \draw[fill] (0,1) circle (.1cm);
    \node at (-0.5,1) {2};

    \draw[fill] (0,0) circle (.1cm);
    \node at (-0.5,0) {1};

    \draw (1,2) -- (0,1);
    \draw (-1,3) -- (-1,2);
    \draw (-1,2) -- (0,1);
    \draw (0,1) -- (0,0);
  \end{tikzpicture}
  \vspace{5mm}

  \begin{minipage}{0.3\textwidth}\centering
    \begin{ytableau}
      \none[1] & *(gray) & *(gray) & *(gray) & *(gray) & *(gray)\\
      \none[2] & & *(gray) & *(gray) & *(gray) & *(gray)\\
      \none[3] & & & *(gray) & *(gray) & *(gray)\\
      \none[4] & & & & *(gray) & *(gray)\\
      \none[5] & & & *(gray) & *(gray) & *(gray) \\
      \none & \none[1] & \none[2] & \none[3] & \none[4] & \none[5]
    \end{ytableau}
  \end{minipage}
  $\xrightarrow{\hspace{5mm}\pi\hspace{5mm}}$
  \begin{minipage}{0.3\textwidth}\centering
    \begin{ytableau}
      \none[1] & *(gray) & *(gray) & *(gray) & *(gray) & *(gray)\\
      \none[2] & *(gray) & *(gray) & *(gray) & *(gray) & *(gray)\\
      \none[3] & *(gray) & *(gray) & *(gray) & *(gray) & *(gray)\\
      \none[4] & *(gray) & *(gray) & *(gray) & *(gray) & *(gray)\\
      \none[5] & & & *(gray) & *(gray) & *(gray) \\
      \none & \none[1] & \none[2] & \none[3] & \none[4] & \none[5]
    \end{ytableau}
  \end{minipage}
  \caption{A poset $P$, and the projection map $\pi:\low_P\to\low_{P^{(m)}}$ for $m=5$.}
  \label{fig:proj_exmp}
\end{figure}

The map $\pi$ induces an action of $\UP$ on $\low_Q$, the orbits of which are easy to analyze.
Define the \dfn{support} of an element $X\in \low_Q$ to be \[\supp(X):=\{x\in Q: X_{m,x}\ne0\}.\]
Each $\UP$-orbit of $\low_Q$ contains precisely one element whose support is an anti-chain in $\lb(m)$.
We can stratify the $\UP$-orbits of $\low_P$ by their image in $\low_Q$ under them map $\pi$.
That is,
\begin{equation}\label{eqn:ugly_sum}
  k(P)=\sum_X\abs{\pi^{-1}(X)\middle/\stab_{\UP}(X)}\tag{$\ast$},
\end{equation}
where the sum is over all elements in $\low_{Q}$ with anti-chain support.

Moreover, if $X,Y\in\low_{Q}$ have the same support $A\in\ac(\lb(m))$, then the corresponding summands for $X$ and for $Y$ in~\eqref{eqn:ugly_sum} are equal.
This can be seen by allowing the diagonal matrices to act on $\low_P$ by conjugation, and noting that for an appropriate choice of diagonal matrix~$\dz$, we have
\[\dz \ts \pi^{-1}(X)\ts \dz^{-1}\. =\. \pi^{-1}(Y)\ts.\]
Furthermore, for the same diagonal matrix $\dz$, we have
\[\dz\ts\stab_{\UP}(X)\ts \dz^{-1}\. =\. \stab_{\UP}(Y)\ts.
\]
Therefore we can sum over a single representative for each anti-chain, and take each summand with multiplicity $(q-1)^{\abs{A}}$.
That is,
\begin{equation}\label{eqn:less_ugly_sum}
  k(P)=\sum_{A\in\ac(\lb(m))}(q-1)^{\abs{A}}\abs{\pi^{-1}(1_A)\middle/\stab_{\UP}(1_A)}\tag{$\ast\ast$}
\end{equation}
where $1_A=\sum_{a\in A} e_{m,x}$, the indicator function on $A$.
Pictorially, we are stratifying the $\UP$-orbits of $\low_P$ by the bottom row in their associated diagram (see Figure~\ref{fig:pictorial}).

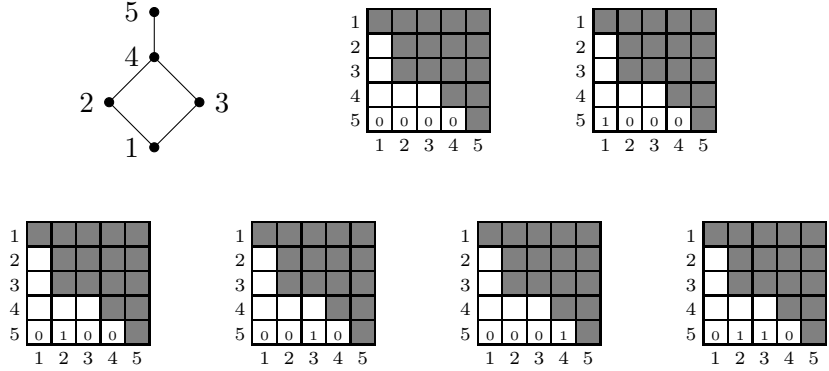
\begin{figure}[h]
  \begin{minipage}{0.2\textwidth}
    \begin{tikzpicture}[scale=0.6]
      \draw[fill] (0,0) circle (.1cm);
      \node at (-0.5,0) {1};

      \draw[fill] (-1,1) circle (.1cm);
      \node at (-1.5,1) {2};

      \draw[fill] (1,1) circle (.1cm);
      \node at (1.5,1) {3};

      \draw[fill] (0,2) circle (.1cm);
      \node at (-0.5,2) {4};

      \draw[fill] (0,3) circle (.1cm);
      \node at (-0.5,3) {5};

      \draw (0,2) -- (1,1) -- (0,0) -- (-1,1) -- (0,2) -- (0,3);
    \end{tikzpicture}
  \end{minipage}
  \ytableausetup{smalltableaux}
  \tiny
  \begin{minipage}{0.2\textwidth}\centering
    \begin{ytableau}
      \none[1] & *(gray) & *(gray) & *(gray) & *(gray) & *(gray)\\
      \none[2] & & *(gray) & *(gray) & *(gray) & *(gray)\\
      \none[3] & & *(gray) & *(gray) & *(gray) & *(gray)\\
      \none[4] & & & & *(gray) & *(gray)\\
      \none[5] & 0 & 0 & 0 & 0 & *(gray) \\
      \none & \none[1] & \none[2] & \none[3] & \none[4] & \none[5]
    \end{ytableau}
  \end{minipage}
  \begin{minipage}{0.2\textwidth}\centering
    \begin{ytableau}
      \none[1] & *(gray) & *(gray) & *(gray) & *(gray) & *(gray)\\
      \none[2] & & *(gray) & *(gray) & *(gray) & *(gray)\\
      \none[3] & & *(gray) & *(gray) & *(gray) & *(gray)\\
      \none[4] & & & & *(gray) & *(gray)\\
      \none[5] & 1 & 0 & 0 & 0 & *(gray) \\
      \none & \none[1] & \none[2] & \none[3] & \none[4] & \none[5]
    \end{ytableau}
  \end{minipage}
  \vspace{7mm}

  \begin{minipage}{0.2\textwidth}\centering
    \begin{ytableau}
      \none[1] & *(gray) & *(gray) & *(gray) & *(gray) & *(gray)\\
      \none[2] & & *(gray) & *(gray) & *(gray) & *(gray)\\
      \none[3] & & *(gray) & *(gray) & *(gray) & *(gray)\\
      \none[4] & & & & *(gray) & *(gray)\\
      \none[5] & 0 & 1 & 0 & 0 & *(gray) \\
      \none & \none[1] & \none[2] & \none[3] & \none[4] & \none[5]
    \end{ytableau}
  \end{minipage}
  \begin{minipage}{0.2\textwidth}\centering
    \begin{ytableau}
      \none[1] & *(gray) & *(gray) & *(gray) & *(gray) & *(gray)\\
      \none[2] & & *(gray) & *(gray) & *(gray) & *(gray)\\
      \none[3] & & *(gray) & *(gray) & *(gray) & *(gray)\\
      \none[4] & & & & *(gray) & *(gray)\\
      \none[5] & 0 & 0 & 1 & 0 & *(gray) \\
      \none & \none[1] & \none[2] & \none[3] & \none[4] & \none[5]
    \end{ytableau}
  \end{minipage}
  \begin{minipage}{0.2\textwidth}\centering
    \begin{ytableau}
      \none[1] & *(gray) & *(gray) & *(gray) & *(gray) & *(gray)\\
      \none[2] & & *(gray) & *(gray) & *(gray) & *(gray)\\
      \none[3] & & *(gray) & *(gray) & *(gray) & *(gray)\\
      \none[4] & & & & *(gray) & *(gray)\\
      \none[5] & 0 & 0 & 0 & 1 & *(gray) \\
      \none & \none[1] & \none[2] & \none[3] & \none[4] & \none[5]
    \end{ytableau}
  \end{minipage}
  \begin{minipage}{0.2\textwidth}\centering
    \begin{ytableau}
      \none[1] & *(gray) & *(gray) & *(gray) & *(gray) & *(gray)\\
      \none[2] & & *(gray) & *(gray) & *(gray) & *(gray)\\
      \none[3] & & *(gray) & *(gray) & *(gray) & *(gray)\\
      \none[4] & & & & *(gray) & *(gray)\\
      \none[5] & 0 & 1 & 1 & 0 & *(gray) \\
      \none & \none[1] & \none[2] & \none[3] & \none[4] & \none[5]
    \end{ytableau}
  \end{minipage}
  \caption{The stratification of $\low_{P}$, where $P$ is the displayed poset.
  The maximal element is 5, and $\lb_P(5)=\{1,2,3,4\}$.
  Each diagram shown represents the form of an element in $\pi^{-1}(1_A)$ for a different anti-chain $A$.
  The possible anti-chains consist of the empty set, the singletons, and $\{2,3\}$ (displayed in order).
We take the first with multiplicity 1, the next four with multiplicity $(q-1)$, and the last one with multiplicity $(q-1)^2$.}
\label{fig:pictorial}
\end{figure}

The notation in~\eqref{eqn:less_ugly_sum} is quite cumbersome, even after suppressing some of the subscripts.
We make the following definition which keeps track of the essential data.

\begin{defn}
  A \dfn{poset system} is a triple $(P,m,A)$ consisting of a poset~$P$,
  a maximal element $m\in\max(P)$, and an anti-chain $A\in\ac(\lb(m))$.
\end{defn}

Let $S=(P,m,A)$ be a poset system.  By a slight abuse of notation,
we define $k(S)=k(S;q)$ as follows:
\[ k(S) \. := \. \abs{\pi^{-1}(1_A)\middle/\stab_{\UP}(1_A)}\ts,\]
where $\pi=\pi_{P,Q}$ and $Q=P^{(m)}$ as above.
For any poset~$P$, and any $m\in\max(P)$, we may rewrite~\eqref{eqn:less_ugly_sum} in this more condensed notation,
\begin{equation}\label{eqn:main_tool}
  k(P)=\sum_{A\in\ac(\lb(m))}(q-1)^{\abs A}k(P,m,A).\tag{$\circ$}
\end{equation}
This relation is our main tool for computing~$k(U_n)$. We show that
under certain conditions on poset systems~$S$,
there exists a poset $Q$ for which ${k(S)=k(Q)}$.
When such a poset exists, we then recursively apply~\eqref{eqn:main_tool}.

Formally, whenever $k(S)=k(P)$ for a poset~$P$ and poset system~$S$,
we say that $S$ \dfn{reduces to}~$P$, and that $S$ is \dfn{reducbile}.

\begin{rmk}
  If every poset system was reducible, an inductive argument would imply that~$k(P)$ was a polynomial for every poset~$P$.
  This is certainly not the case, as Halasi and P\'alfy have constructed posets for which $k(P)$ is not a polynomial~\cite{HP}.
  However, by adding suitable constraints, we can guarantee that $S$ is reducible.
\end{rmk}

\begin{lem}\label{lem:remove_max}
  Let $S=(P,m,A)$ be a poset system such that there exists no pair of elements $(a,x)\in A\times P$ for which $a\prec x\prec m$.
  Then $S$ reduces to $P-m$.
\end{lem}

\begin{proof}
  We begin by showing that the entire group $\UP$ stabilizes $1_A$.
  Let $\alpha\in\fq^\times$, and let $E=\E_{x,y}(\alpha)$ be a generator of $\UP$.
  If $x\not\in A$, then it is easy to see that $K_E(1_A)=1_A$.
  On the other hand, if $x\in A$, then by assumption $y\not\prec m$.
  Therefore, we have $K_E(1_A)=1_A-e_{m,y}=1_A$, as $e_{m,y}$ is trivial in $\low_Q$.

  From Proposition~\ref{prop:generators}, we know that each of the generators $\E_{x,y}(\alpha)$ of $\UP$ is either a generator of $U_{\hspace{-0.6mm}P-m}$, or of the form $\E_{x,m}(\alpha)$ for $\alpha\in\fq^\times$.
  Because each generator of the form $\E_{x,m}(\alpha)$ acts trivially on $\low_P$, we have
  \[k(S)=\abs{\pi^{-1}(1_A)/\stab_{\UP}(1_A)}=\abs{\pi^{-1}(1_A)/U_{\hspace{-0.6mm}P-m}}.\]

  Now every element of $U_{\hspace{-0.6mm}P-m}$ acts trivially on row $m$ (the $\fq$-linear span of $e_{m,x}$).
  Simply removing this row yields the co-adjoint action of $U_{\hspace{-0.6mm}P-m}$ on $\low_{P-m}$, so \[k(S)=\abs{\pi^{-1}(1_A)/U_{\hspace{-0.6mm}P-m}}=\abs{\low_{P-m}/U_{\hspace{-0.6mm}P-m}}=k(P-m)\] as desired.
\end{proof}

\begin{lem}\label{lem:normal_conj}
  Let $(P,m,A)$ be a poset system, and suppose that $a,b\in A$ such that
  \[\ub(a)\supseteq\ub(b)\text{ and }\lb(a)\subseteq\lb(b).\]
  Then \ts $k(P,m,A)=k(P,m,A-\{b\})$.
\end{lem}

\begin{proof}
  Let \ts $\Phi:\low_P\to\low_P$ \ts denote conjugation by \ts $E=\E_{a,b}(1)$.
  Note that \ts $\E\not\in\UP$, since $a$ and $b$ are incomparable.
  However, $E$ normalizes $\UP$, and so the map~$\Phi$ is well-defined.
  As a slight abuse of notation, we also use $\Phi$ to denote the conjugation map~$\Phi:\UP\to\UP$ given by $\Phi(g)=EgE^{-1}$.
  It is now a triviality that for $X\in\low_P$ and $g\in\UP$, we have \[\Phi(K_g(X))=K_{\Phi(g)}(\Phi(X)).\]

  Now let $Q=P^{(m)}$ and $\pi=\pi_{P,Q}$.
  Pushing $\Phi$ through $\pi$ to an action of $\low_Q$, we have
  \[\Phi(1_A)=E(1_A)E^{-1}=1_A-e_{m,b}=1_{A-\{b\}}.\]
  Moreover, as $\Phi$ commutes with $\pi$, we have $\Phi(\pi^{-1}(1_A))=\pi^{-1}(1_{A-\{b\}})$.
  Lastly, note that $\Phi(\stab_{\UP}(1_A))=\stab_{\UP}(1_{A-\{b\}})$.
  Thus, we have \[k(P,m,A)=\abs{\pi^{-1}(1_A)\middle/\stab_{\UP}(1_A)}=\abs{\pi^{-1}(1_A)\middle/\stab_{\UP}(1_A)}=k(P,m,A-\{b\}),\] as desired.
\end{proof}

\subsection{The operator $\D$}
Let $S=(P,m,A)$ be a poset system.
Define $\D(S)$ to be a poset obtained from $P$ by removing relations $a\prec x$ whenever the following two criteria hold:

\begin{enumerate}
  \item $a\in A$, and $a\prec x\prec m$.
  \item If $a'\in A$ and $a'\prec x$, then $a'=a$.
\end{enumerate}
Stated more concisely, the set of pairs of related elements in $\D(S)$ is given by \[\rel(\D(S))=\rel(P)\setminus\{(a,x): a\prec x\prec m,\ \abs{A\cap\lb(x)}=1\}.\]

In figures~\ref{fig:D_example1} and~\ref{fig:D_example2}, we provide examples of poset systems $S$ and the application of the operator~$\D(S)$.
Poset systems are shown graphically as the Hasse diagram of their underlying poset with special marked elements.
Generic elements of $P$ will be denoted by ``$\bullet$'' as with they normally are in the Hasse diagram of a poset.
The elements of the anti-chain $A$ will be denoted by ``$\circ$.''
The maximal element $m$ will be denoted by ``{\tiny$\square$}.''

\begin{figure}[h]
  \centering
  \begin{tikzpicture}[scale=0.6]
    \draw (0,4) -- (0,0);

    \draw[draw=black, fill=white] (-0.1,3.9) rectangle (0.1,4.1);
    \node at (-0.5,4) {5};

    \draw[fill] (0,3) circle (.1cm);
    \node at (-0.5,3) {4};

    \draw[draw=black, fill=white] (0,2) circle (.1cm);
    \node at (-0.5,2) {3};

    \draw[fill] (0,1) circle (.1cm);
    \node at (-0.5,1) {2};

    \draw[fill] (0,0) circle (.1cm);
    \node at (-0.5,0) {1};

    \draw[fill] (10,3) circle (.1cm);
    \node at (9.5,3) {5};

    \draw[fill] (10,2) circle (.1cm);
    \node at (9.5,2) {3};

    \draw[fill] (11,2) circle (.1cm);
    \node at (11.5,2) {4};

    \draw[fill] (10,1) circle (.1cm);
    \node at (9.5,1) {2};

    \draw[fill] (10,0) circle (.1cm);
    \node at (9.5,0) {1};

    \draw (11,2) -- (10,1);
    \draw (10,3) -- (11,2);
    \draw (10,3) -- (10,0);
    \draw (10,1) -- (10,0);

  \end{tikzpicture}
  \caption{The poset system $S=(\chain 5,5,\{3\})$ shown on the left and the poset~$\D(S)$ shown on the right.}
  \label{fig:D_example1}

  \vspace{1cm}
  \begin{tikzpicture}[scale=0.6]
    \draw (1,4) -- (0,3);
    \draw (-1,4) -- (0,3);
    \draw (0,3) -- (0,0);

    \draw[draw=black, fill=white] (-1.1,3.9) rectangle (-0.9,4.1);
    \node at (-1.5,4) {6};

    \draw[fill] (1,4) circle (.1cm);
    \node at (1.5,4) {5};

    \draw[fill] (0,3) circle (.1cm);
    \node at (-0.5,3) {4};

    \draw[draw=black, fill=white] (0,2) circle (.1cm);
    \node at (-0.5,2) {3};

    \draw[fill] (0,1) circle (.1cm);
    \node at (-0.5,1) {2};

    \draw[fill] (0,0) circle (.1cm);
    \node at (-0.5,0) {1};

    \draw[fill] (9,3) circle (.1cm);
    \node at (8.5,3) {6};

    \draw[fill] (10,3) circle (.1cm);
    \node at (10.5,3) {5};

    \draw[fill] (9,2) circle (.1cm);
    \node at (8.5,2) {4};

    \draw[fill] (10,2) circle (.1cm);
    \node at (10.5,2) {3};

    \draw[fill] (10,1) circle (.1cm);
    \node at (10.5,1) {2};

    \draw[fill] (10,0) circle (.1cm);
    \node at (10.5,0) {1};

    \draw (9,3) -- (9,2);
    \draw (9,3) -- (10,2);
    \draw (9,2) -- (10,3);
    \draw (10,2) -- (10,3);
    \draw (9,2) -- (10,1);
    \draw (10,2) -- (10,1);
    \draw (10,1) -- (10,0);

  \end{tikzpicture}

  \caption{The poset system $S=(P, 6, \{3\})$ shown on the left and the poset~$\D(S)$ shown on the right.}
  \label{fig:D_example2}
\end{figure}

\begin{lem}\label{lem:apply_d}
  For any poset system $S=(P,m,A)$, we have $k(S)=k(\D(S),m,A)$.
\end{lem}

\begin{proof}
  Let $Q=P^{(m)}$.
  Not only is $Q$ a subposet of $P$, but it is also a subposet of $\D(S)$.
  Therefore every element of $A$ is less than $m$ in $\D(S)$ as well as in $P$, so the poset system $(\D(S),m,A)$ is well-defined.

  We first show that $\stab_{\UP}(1_A)=\stab_{U_{\D(S)}}(1_A)$.
  Clearly $\stab_{U_{\D(S)}}(1_A)\le\stab_{\UP}(1_A)$, so to show equality, it suffices to show that the two stabilizers have the same cardinality.
  Let~$\Omega_P$ denote the $\UP$-orbit of $\low_Q$ containing~$1_A$, and let $\Omega_{\D(S)}$ denote the $U_{\D(S)}$-orbit of $\low_{Q}$ containing~$1_A$.
  By the orbit-stabilizer theorem, it is enough to show that \[\frac{\abs{\UP}}{\abs{U_{\D(S)}}} = \frac{\abs{\Omega_P}}{\abs{\Omega_{\D(S)}}}.\]

  It is immediate from the definition of pattern groups that $\abs{\UP}=q^{\abs{\rel(P)}}$.
  Thus, we have $\abs{\UP}/\abs{U_{\D(S)}}=q^{\abs{R}}$, where $R=\rel(P)\setminus\rel(\D(S))$.
  We may characterize $R$ in a different way:
  \[R=\{(a,x)\in A\times \lb(m):\text{$a$ is the unique element of $A$ below $x$}\}.\]
  For pairs $(a,x)\in R$, the element $a\in A$ is uniquely defined by $x$, and so $R$ is in bijection with the set \[R'=\{x\prec_P m: \abs{\lb(x)\cap A}=1\}.\]

  We now turn to the $\UP$-orbits $\Omega_P$ and $\Omega_{\D(S)}$ in $\low_Q$.
  Certainly $X_{m,a}=1$ for each $a\in A$, and $X_{m,x}=0$ if $x\not\succ_P a$ for all $a\in A$.
  If, on the other hand, there does exist some $a\in A$ for which $a\prec_P x$, then by conjugation one can obtain any value at $X_{m,x}$.
  Specifically, note that for~$E=\E_{a,x}(\alpha)$, we have $K_E(X)=X-\alpha e_{m,x}$.
  It follows that

  \[\frac{\abs{\Omega_P}}{\abs{\Omega_{\D(S)}}}\. = \. q^{|R_1| \ts - \ts |R_2|}\ts, \ \ \. \text{where}\]

 \[R_1 \ts = \ts \{x\prec_P m~:~a\prec_P x\text{ for some }a\in A\} \ \, \text{and} \ \,  R_2 \ts = \ts
 \{x\prec_P m~:~a\prec_{\D(S)} x\text{ for some }a\in A\}.\]

  From the definition of $\D(S)$, we have $a\prec_{\D(S)} x$ if and only if there is more than one element of~$A$ which is less than $x$ in $P$.
  Hence, \[\frac{\abs{\Omega_P}}{\abs{\Omega_{\D(S)}}}=q^{\#\{x~:~\abs{\lb(x)\cap A}=1\}}=q^{\abs{R'}}.\]

  This proves that $\stab_{\UP}(1_A)=\stab_{U_{\D(S)}}(1_A)$.
  For the remainder of the proof, we let $G$ denote both of these groups.
  We are now left to show that $\pi_{P,Q}^{-1}(1_A)$ and $\pi_{\D(S),Q}^{-1}(1_A)$ are isomorphic $G$-sets.
  To do so, we need to construct a map between these two sets which preserves $G$-orbits.
  There is a natural choice for such a map:
  The canonical projection $\pi_{P,\D(S)}:\low_P\to\low_{\D(S)}$ restricts to \[\rho:\pi_{P,Q}^{-1}(1_A)\longrightarrow\pi_{\D(S),Q}^{-1}(1_A).\]

  We now argue that $\rho$ preserves $G$-orbits.
  More precisely, we claim that for all \mbox{$X,Y\in\pi_{P,Q}^{-1}(1_A)$,} the elements $X$ and $Y$ belong to the same $G$-orbit if and only if $\rho(X)$ and $\rho(Y)$ belong to the same $G$-orbit.

  Because $\rho$ respects the co-adjoint action, it is clear that $\rho(X)$ and $\rho(Y)$ belong to the same $G$-orbit whenever $X$ and $Y$ belong to the same $G$-orbit.
  In the other direction, suppose \mbox{$\rho(X)=K_g(\rho(Y))$} for some $g\in G$.
  Then $X-K_g(Y)\in\ker\rho$.
  It is easy to see that
  \[\ker\rho\, = \. \bigoplus_{(a,x)\in R}\fq e_{x,a}\hspace{1mm}.\]
  Indeed, the pairs $(a,x)\in R$ are precisely the pairs of elements
  for which $a\prec_P x$ but $a\not\prec_{\D(S)}x$, so linear combinations of the $e_{x,a}$ exactly the elements which are projected away by $\rho$. Now let $(a,x)\in R$, and let $E=\E_{x,m}(\alpha)$.
  For $Z\in\pi_{P,Q}^{-1}(1_A)$, we have \[K_h(Z)=Z+\alpha e_{a,x}.\]
  Thus, if two elements of $\pi_{P,Q}^{-1}(1_A)$ differ by an element of $\ker\rho$, they must belong to the same $G$-orbit.
  In particular, $X$ and $K_g(Y)$ belong to the same $G$-orbit.
  This proves \[k(S)=\abs{\pi_{P,Q}^{-1}(1_A)/G}=\abs{\pi_{\D(S),Q}^{-1}(1_A)/G}=k(\D(S),m,A),\] which completes the proof.
\end{proof}

\begin{lem}\label{lem:antichain_chain}
  Let $S=(P,m,A)$ be a poset system with $A=\{a_1,\dots,a_k\}$ such that
  \begin{align*}\label{eqn:lbub}
    \lb_P(a_1)\subseteq\lb_P(a_2)\subseteq\cdots\subseteq\lb_P(a_k)\text{ and}\\
    \ub_P(a_1)\subseteq\ub_P(a_2)\subseteq\cdots\subseteq\ub_P(a_k).\hspace{5mm}
  \end{align*}
  Further suppose that $m$ is the unique maximum above $a_1$.
  Then $S$ is reducible.
\end{lem}

\begin{proof}
  We proceed by induction on $\abs{A}$.
  If $A=\emp$, then $k(S)=k(P-m)$ by Lemma~\ref{lem:remove_max}.
  If $\abs A=1$, then $k(S)=k(\D(S)-m)$ by lemmas~\ref{lem:apply_d} and~\ref{lem:remove_max} applied in succession.

  Now suppose the result holds whenever the anti-chain has fewer than $k$ elements, and let $\abs{A}=k$.
  Applying Lemma~\ref{lem:apply_d}, we have $k(S)=k(\D(S),m,A)$.
  Let
  \[R \. := \. \rel(P)\setminus\rel(\D(S))\ts,
  \]
  and note that because $m$ is the unique maximum in~$P$, we have
  \[R \. = \. \left\{(a,x)\in A\times P: \lb(x)\cap A=\{a\}\right\}\ts.
  \]
  If $(a_i,x)\in R$, then $a_i\prec_P x$, and for all $j\ne i$, it must be that $a_j\not\prec_P x$.
  Thus, if $(a_i,x)\in R$, it must be that $i=k$, and so \[R=\{(a_k,x): x\in\ub_P(a_k)\setminus\ub_P(a_{k-1})\}.\]
  Therefore $\ub_{\D(S)}(a_k)=\ub_{\D(S)}(a_{k-1})$, and so $(\D(S),m,A)$ satisfies the hypotheses of
  Lemma~\ref{lem:normal_conj}.  This tells us that $k(\D(S),m,A)=k(\D(S),m,A-\{a_k\})$.
  By inductive hypothesis, there exists a poset $Q$ for which $k(\D(S),m,A-\{a_k\})=k(Q)$.
  Stringing these equalities together yields~$k(S)=k(Q)$, as desired.
\end{proof}

\subsection{Reduction of \Y-posets}
With suitable constraints on the poset, we may obtain a recurrence relation for the number of conjugacy classes in its pattern group.
One such constraint is as follows.
Define the poset $\Y$ as in Figure~\ref{fig:y_poset}.
Recall that a poset is \dfn{\Y-free} if it does not have the poset $\Y$ as an induced subposet.

\begin{figure}[h]
  \centering
  \begin{tikzpicture}[scale=0.6]
    \draw (0,1) -- (0,2);
    \draw (-1,0) -- (0,1) -- (1,0);

    \draw[fill] (0,2) circle (.1cm);
    \draw[fill] (0,1) circle (.1cm);
    \draw[fill] (1,0) circle (.1cm);
    \draw[fill] (-1,0) circle (.1cm);
  \end{tikzpicture}
  \caption{The poset $\Y$.}
  \label{fig:y_poset}
\end{figure}

\begin{thm}\label{thm:y_free}
  Let $P$ be a \Y-free poset, and let $m\in\max(P)$.
  Then \[k(P)=\sum_{S=(P,m,A)}(q-1)^{\abs{A}}k(\D(S)-m).\]
\end{thm}

\begin{proof}
  Let $S=(P,m,A)$ be a poset system.
  In light of~\eqref{eqn:main_tool}, it suffices to show that \[k(S)=k(\D(S)-m).\]
  By Lemma~\ref{lem:apply_d}, we see that $k(S)=k(\D(S),m,A)$.
  We claim that if $P$ is \mbox{\Y-free}, then $\D(S)$ has no element $x$ for which $a\prec_{\D(S)} x\prec_{\D(S)} m$.
  Once this claim is established, Lemma~\ref{lem:remove_max} proves that ${k(S)=k(\D(S)-m)}$.
  Suppose for the sake of contradiction that \mbox{$x\in\D(S)$}, and $a\in A$ such that \[{a\prec_{\D(S)}x\prec_{\D(S)}m}.\]
  Because $\D(S)$ is obtained from $P$ by removing relations, certainly $a\prec_P x\prec_P m$.
  Moreover, because $a\prec_P x$ was not removed, we know that $\abs{A\cap\lb(x)}>1$.
  Thus, there must be some other $b\in A$ with $b\prec_P x$.
  Now $\{a,b,x,m\}$ induces a copy of \Y~in $P$, which is a contradiction.
\end{proof}

\begin{rmk}
  Theorem~\ref{thm:y_free} did not use the full strength of the \Y-freeness condition.
  It is only necessary that $P$ be \Y-free below a single maximal element.
  Hence, we have the following strengthening of the theorem.
\end{rmk}

\begin{thm}\label{thm:y_free_stronger}
  Let $P$ be a poset, and suppose that there exists some $m\in\max(P)$ such that the poset induced on $\{x: x\preceq_P m\}$ is \Y-free.
  Then \[ k(P)=\sum_{S=(P,m,A)}(q-1)^{\abs{A}}k(\D(S)-m).\]
\end{thm}


\medskip

\subsection{Interval posets}
In a different direction, we consider interval posets.
Given a collection of closed intervals $I_k=[\ell_k,r_k]$ in $\rr$,
one can define a partial order called the \dfn{interval order} on $\{I_k\}$
by declaring $I_j\preceq I_k$ whenever $r_j\le \ell_k$.  An \dfn{interval poset}
is a poset which is the interval order of some family of intervals on a line.
The class of interval posets is well studied (see e.g.~\cite{Tro}),
and has several equivalent characterizations.  For our purposes,
the important properties of interval poset will be items~(3) and~(4)
in the following theorem.

\begin{thm}[Theorem 3.2 in~\cite{Tro}]
\label{thm:interval_equiv}
  For a poset $P$, the following are equivalent:
  \begin{enumerate}
    \item $P$ is an interval poset,
    \item $P$ is $(\chain2\amalg\chain2)$-free,
    \item the collection of sets $\ub(x)$ for $x\in P$ is totally ordered by inclusion,
    \item and the collection of sets $\lb(x)$ for $x\in P$ is totally ordered by inclusion.
  \end{enumerate}
\end{thm}

From here we have the following positive result.

\begin{thm}
  Every interval poset with a unique maximal element is reducible.
\end{thm}

\begin{proof}
  From~\eqref{eqn:main_tool}, it suffices to show that every poset system $S=(P,m,A)$ is reducible.
  We do so by induction on $\abs{A}$. If $\abs{A}=0$ then the result follows from Lemma~\ref{lem:remove_max}.
  If $\abs{A}=1$, the result follows from lemmas~\ref{lem:apply_d} and~\ref{lem:remove_max} applied in succession.
  Otherwise, suppose that $\abs{A}\ge 2$.
  If there exist $a,b\in A$ which satisfy the conditions of Lemma~\ref{lem:normal_conj}, then $k(S)=k(P,m,A-\{b\})$, and the inductive hypothesis proves the claim.
  We may therefore assume that for every $a,b\in A$, whenever $\lb(a)\subseteq\lb(b)$ we also have $\ub(b)\nsubseteq\ub(a)$.
  Recall that in an interval poset, the sets $\lb(x)$ are totally ordered by inclusion, so we order the elements of $A=\{a_1,\dots,a_k\}$ such that
  \[\lb(a_1)\subseteq\lb(a_2)\subseteq\cdots\subseteq\lb(a_k).\]
  For each $i<j$, we know that $\ub(a_j)\nsubseteq\ub(a_i)$.
  However, in an interval poset, the sets $\ub(x)$ are also totally ordered by inclusion.
  We conclude that
  \[\ub(a_1)\subseteq\ub(a_2)\subseteq\cdots\subseteq\ub(a_k),\]
  and the result follows from Lemma~\ref{lem:antichain_chain}.
\end{proof}

\medskip

%
%

\section{Embedding}\label{sec:embedding}
\subsection{Embedding sequences}
Consider an attempt to compute $k(U_n)=k(\chain n)$ by recursively applying~\eqref{eqn:main_tool} along with the other tools developed in Section~\ref{sec:combo_tools}.
If a poset system $S$ appears in a computation and is reducible to a poset $P$, we can replace $k(S)$ with $k(P)$, and compute $k(P)$, applying~\eqref{eqn:main_tool} again.
We show that for every poset $P$, one can take $n$ sufficiently large so that $k(P)$ appears in the recursive expansion of $k(U_n)$.
With the following definition, we make this statement precise in Theorem~\ref{thm:chain_univ}.

\begin{defn}
  We say that a poset $P$ \dfn{strongly embeds}\footnote{We define a weaker notion which we call \emph{embedding} later on in Definition~\ref{defn:embed}.} into a poset $Q$ if there exists a sequence of poset systems $S_1,\dots, S_n$ with $S_i=(P_i,m_i,\{a_i\})$, such that
  \begin{enumerate}
    \item $P_0=P$,
    \item $P_n=Q$,
    \item for $0\le i<n$, we have $P_i\cong\D(S_{i+1})-m_{i+1}$.
  \end{enumerate}
  When $P$ strongly embeds into $Q$, we write $P\embeds Q$.
  The sequence \[P=P_0\embeds P_1\embeds\cdots\embeds P_{n-1}\embeds P_n=Q\] is called a \dfn{strong embedding sequence}.
  When we wish to signify that the strong embedding sequence has length~$n$, we write $P\embedsin{n}Q$.
\end{defn}

Note that the anti-chains in each poset system are required to have exactly one element.
Thus, lemmas~\ref{lem:remove_max} and~\ref{lem:apply_d} can be applied, and $k(P_i)=k(S_{i+1})$.
The following observations regarding strong embedding are easy.
\begin{prop}
  Let $P$, $Q$, and $R$ be posets such that $P\embedsin{k} Q$. Then
  \begin{enumerate}
    \item $R+P \embedsin{k} R+Q$, and
    \item $R\amalg P \embedsin{k} R\amalg Q$.
  \end{enumerate}
\end{prop}

The next few lemmas are technical, so we provide an outline of our methods for showing that every poset strongly embeds into a chain.
First, Lemma~\ref{lem:add_top} tells us that if we have a poset $P$ sitting inside a larger poset $P+\chain k$, it is safe to focus just on $P$.
That is, any strong embedding of $P$ into a chain can be transformed into a strong embedding of $P+\chain k$ into an even larger chain.
With this in mind, we may safely assume that $P$ does not have a unique maximum.

Next, Lemma~\ref{lem:max_el} proves that we can take a maximal element $m$ of $P$ and connect it to each of the other elements in $P$.
The result will be a poset which has a chain sitting atop it which can safely be ignored.

Finally, the content of Theorem~\ref{thm:chain_univ} applies Lemma~\ref{lem:max_el} inductively, proving that each poset strongly embeds into a chain.

\begin{lem}\label{lem:add_top}
  Let $P$, $Q$, and $R$ denote posets, and suppose $P\embedsin{k} Q$.
  Then we have \[P+R\embedsin{2k} Q+R+\chain k.\]
\end{lem}

\begin{proof}
  We proceed by induction on $k$.
  We first show the result for $P\embedsin{1}Q$.
  Let $(Q,m,\{a\})$ be a poset system for which $P\cong \D(Q,m,\{a\})-m$.
  Then \[\rel(P)=\rel(Q-m)\setminus\{(a,x): a\prec_Q x\prec_Q m\}.\]
  We define poset systems $S_1$ and $S_2$ to yield a strong embedding sequence for $P+R\embeds Q+R+\chain1$.
  We work backwards from $Q+R+\chain1$, first defining $S_2$, then defining $S_1$ in terms of $S_2$.

  Let $m'$ denote the unique maximal element in $Q+R+\chain1$, and define
  \begin{align*}
    S_2 &= (Q+R+\chain1, m', \{m\})\text{ and}\\
    S_1 &= (\D(S_2)-m', m, \{a\}).
  \end{align*}
  We aim to show that $\D(S_1)-m\cong P+R$.
  To this end, we begin with $Q+R+\chain1$ and follow backwards through the strong embedding sequence to determine which relations were removed.
  First, for $\D(S_2)-m'$, the relations removed were all the relations of the form $(m,r)$ for $r\in R$.
  It follows that $m$ is maximal in $\D(S_2)-m'$.

  Next $\D(S_1)-m$ removes all of the relations $(a,x)$ where $a\prec_Q x\prec_Q m$.
  The result is that $P+R$ and $\D(S_1)-m$ have precisely the same relations and are therefore isomorphic posets.
  This proves that $P+R\embedsin{2}Q+R+\chain1$, which concludes the base case.

  Assume that for all posets $P$, $Q$, and $R$, whenever $P\embedsin{k} Q$ we have $P+R\embedsin{2k} Q+R+\chain k$.
  Suppose we have posets $P$ and $Q$ for which $P\embedsin{k+1} Q$.
  Write the strong embedding sequence \[P=P_0\embedsin{k} P_k\embedsin{1} Q.\]
  By the inductive hypothesis $P+R\embedsin{2k}P_k+R+\chain{k}$.
  Furthermore, because $P_k\embedsin{1} Q$, the base case shows us that \[P_k+(R+\chain{k})\embedsin{2}Q+(R+\chain{k})+\chain1=Q+R+\chain{k+1}.\]
  Together, we have $P+R\embedsin{2k+2}Q+R+\chain{k+1}$, which completes the induction.
\end{proof}

\begin{lem}\label{lem:max_el}
  Let $P$ be a poset, and let $m\in\max(P)$.
  Then \[P\embedsin{k} (P-m)+\chain{k+1},\]
  where $k=\abs{P}-\abs{\lb_P(m)}-1$.
\end{lem}

\begin{proof}
  Let $X=\{x: x\not\preceq_P m\}$, and note that $\abs{X}=k$.
  Order the elements of \mbox{$X=\{x_1,x_2,\dots,x_k\}$} according to some reverse linear extension of $P$, so that if $x_i\preceq_P x_j$, then $i\ge j$.

  Let $Q_0=(P-m)+\chain{\abs{X}+1}$, and label the elements in \[\chain{\abs{X}+1}=\{m<p_k<p_{k-1}<\cdots< p_1\}.\]
  For $1\le i\le k$, define $Q_i$ recursively as $Q_i=\D(Q_{i-1},p_i,\{x_i\})-p_i$.

  The relations removed from $Q_i$ are simple to describe:
  \[\rel(Q_{i-1})\setminus\rel(Q_i)=\{(x_i,p_j): i+1\le j\le k\}\cup\{(x_i,m)\}.\]
  Note that $Q_k$ is a poset which has $p_1,\dots,p_k$ removed.
  Thus, the fact that we removed the relations $\{(x_i,p_j): i+1\le j\le k\}$ from $Q_{i-1}$ to obtain $Q_i$ is not relevant.
  However, we did remove $(x_i,m)$ for each $i$.
  By the definition of $X$, we have the equality $Q_k=P$.
  Thus, we have constructed a strong embedding sequence \[P=Q_k\embeds Q_{k-1}\embeds\cdots\embeds Q_0=(P-m)+\chain{\abs{X}+1},\] which proves the result.
\end{proof}

\begin{thm}\label{thm:chain_univ}
  Every poset strongly embeds into a chain.
  Specifically, $P\embeds \chain{\abs{P}^2-2\abs{\rel(P)}}$.
\end{thm}

\begin{proof}
  Let $F(P)$ denote the set of elements which are not comparable to every element in $P$.
  We proceed by induction on $\abs{F(P)}$.
  If $F(P)=\emp$, then $P$ is a chain and the result is trivial.

  Otherwise, let $m\in F(P)$ be maximal amongst elements of $F(P)$.
  As every element of $\ub(m)$ is comparable to every element in $P$, the elements in $\ub(m)$ are totally ordered.
  Thus, we may disect $P$ into \[P=P_0+\chain\ell,\] where $\ell=\abs{\ub(m)}$, and where $m\in\max(P_0)$.

  By Lemma~\ref{lem:max_el}, we know that \[P_0\embedsin{k} (P_0-m)+\chain{k+1},\] where $k=\abs{P_0}-\abs{\lb_P(m)}-1$.
  Applying Lemma~\ref{lem:add_top}, we see that \[P=P_0+\chain\ell\embedsin{2k}(P_0-m)+\chain{2k+\ell+1}.\]
  Let $Q=(P_0-m)+\chain{2k+1+\ell}$.
  Note that $F(Q)=F(P)\setminus\{m\}$, and so by inductive hypothesis,
  \[P\embeds Q\embeds \chain{\abs{Q}^2-2\abs{\rel{Q}}}.\]
  It now suffices to show that $\abs{Q}^2-\abs{P}^2=2\abs{\rel(Q)}-2\abs{\rel(P)}$.
  To this end, note that \mbox{$\abs{Q}=\abs{P}+2k$}, and so $\abs{Q}^2-\abs{P}^2=4k(k+\abs{P})$.

  We now express both $\abs{\rel(Q)}$ and $\abs{\rel(P)}$ in terms of $\abs{\rel(P_0-m)}$ by conditioning each pair of related elements on whether or not each element of the pair is contained in $P_0-m$. We have
  \begin{align*}
    2\abs{\rel(P)} &= 2\abs{\rel(P_0-m)}+2\abs{\lb_P(m)}+2\ell\abs{P_0}+l(l-1),\text{ and}\\
    2\abs{\rel(Q)} &= 2\abs{\rel(P_0-m)}+2(2k+\ell+1)(\abs{P_0}-1)+(2k+\ell+1)(2k+\ell).
  \end{align*}

  Recalling that $\abs{\lb_P(m)}=\abs{P_0}-k-1$ and simplifying, we have \[2\abs{\rel(Q)}-2\abs{\rel(P)} = 4k(\abs{P_0}+\ell+k)=4k(k+\abs{P}),\]
  which completes the proof.
\end{proof}

\subsection{Consequences for $U_n$}\label{subsec:un_consequences}
Recall that Halasi and P\'alfy proved the existence of a poset~$P$
for which $k(P)$ is not a polynomial~\cite{HP}.  We modify
their construction we obtained a 13-element poset $P_\diamond$ shown in
Figure~\ref{fig:non_poly_poset}, such that $k(P_\diamond)$ is not a
polynomial in~$q$ (c.f.~\S\ref{ssec:fin-rems-posets}).
Using Lemma~3.1 of~\cite{HP}, we have computed $k(P_\diamond)$.

\begin{figure}[t]
  \centering
  \begin{tikzpicture}[scale=0.45]
    \draw[fill] (-3,0) circle (.1cm);
    \draw[fill] (-1,0) circle (.1cm);
    \draw[fill] (1,0) circle (.1cm);
    \draw[fill] (3,0) circle (.1cm);

    \draw[fill] (-5,2) circle (.1cm);
    \draw[fill] (-3,2) circle (.1cm);
    \draw[fill] (-1,2) circle (.1cm);
    \draw[fill] (1,2) circle (.1cm);
    \draw[fill] (3,2) circle (.1cm);
    \draw[fill] (5,2) circle (.1cm);

    \draw[fill] (-2,4) circle (.1cm);
    \draw[fill] (0,4) circle (.1cm);
    \draw[fill] (2,4) circle (.1cm);

    \draw (5,2) -- (3,0);
    \draw (5,2) -- (1,0);
    \draw (5,2) -- (2,4);
    \draw (5,2) -- (0,4);

    \draw (-1,2) -- (-1,0);
    \draw (-1,2) -- (1,0);
    \draw (-1,2) -- (-2,4);
    \draw (-1,2) -- (0,4);

    \draw (1,2) -- (-3,0);
    \draw (1,2) -- (1,0);
    \draw (1,2) -- (-2,4);
    \draw (1,2) -- (2,4);

    \draw (-3,2) -- (-3,0);
    \draw (-3,2) -- (-1,0);
    \draw (-3,2) -- (2,4);
    \draw (-3,2) -- (0,4);

    \draw (-5,2) -- (-3,0);
    \draw (-5,2) -- (3,0);
    \draw (-5,2) -- (-2,4);
    \draw (-5,2) -- (0,4);

    \draw (3,2) -- (-1,0);
    \draw (3,2) -- (3,0);
    \draw (3,2) -- (-2,4);
    \draw (3,2) -- (2,4);
  \end{tikzpicture}
  \caption{A 13-element poset~$P_\diamond$, for which $k(P_\diamond)\notin \zz[q]$.}
  \label{fig:non_poly_poset}
\end{figure}
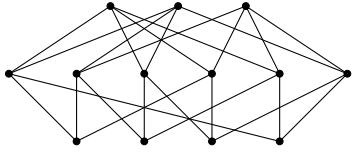


\begin{prop}
  Let $P_\diamond$ denote the poset shown in Figure~\ref{fig:non_poly_poset}. Then
  \begin{align*}
    k(P_0) \, & =  \, 1 + 36 \ts t + 582\ts t^2 + 5628\ts t^3 + 36601\ts t^4 + 170712\ts t^5 + 594892\ts t^6\\
    &\, \hspace{5mm} + 1593937\ts t^7 + 3355488\ts t^8 + 5646608\ts t^9 + 7705410\ts t^{10}\\
    &\, \hspace{5mm} + 8631900\ts t^{11} + 8023776\ts t^{12} + 6248381\ts t^{13} + 4111322\ts t^{14}\\
    &\, \hspace{5mm} + 2302222\ts t^{15} + 1102490\ts t^{16} + 451836\ts t^{17} + 157555\ts t^{18}\\
    &\, \hspace{5mm} + 46042\ts t^{19} + 10971\ts t^{20} + 2040\ts t^{21} + 276\ts t^{22} + 24\ts t^{23} +\ts t^{24}\\
    &\, \hspace{5mm} + \delta(q)\cdot t^{12}(t+2)^6\,,
  \end{align*}
  where \ts $t=q-1$ \ts and
  \[\delta(q)=\begin{cases}
      2 & \text{if $q$ is odd\ts,}\\
      1 & \text{otherwise\ts.}
  \end{cases}\]
\end{prop}

%

This proposition proves Theorem~\ref{thm:hp-new}. Now, it follows from Theorem~\ref{thm:chain_univ}
that $P_\diamond\embeds\chain{97}$.  However, the strong embedding sequence can be made more
efficient by weakening our definition.

\begin{defn}\label{defn:embed}
  A poset $P$ \dfn{embeds} into a poset $Q$ if there exists a sequence of poset systems $S_1,\dots,S_n$ with $S_i=(P_i,m_i,A_i)$, such that
  \begin{enumerate}
    \item $P_0=P$,
    \item $P_n=Q$, and
    \item for $0\le i < n$, we have $k(P_i)=k(S_{i+1})$.
  \end{enumerate}
  When $P$ embeds into $Q$, we write $P\wkemb Q$.
  The sequence \[P=P_0\wkemb P_1\wkemb\cdots\wkemb P_{n-1}\wkemb P_n=Q\] is called an \dfn{embedding sequence}.
  When we wish to signify that the embedding sequence has length $n$, we write $P\wkembin{n}Q$.
\end{defn}

Note that if $P\embeds Q$, then $P\wkemb Q$ as well.
One of the tools we are able to use with embeddings which was not possible with the stricter notion of strong embeddings is the fact that $k(P^*)=k(P)$.
This is the fact that we will exploit to show that $P_\diamond\wkemb\chain{59}$ in Proposition~\ref{prop:more_efficient}.

\begin{lem}\label{lem:two_chains}
  For nonnegative integers $a$ and $b$, we have $\chain a\amalg\chain b\embeds\chain{2a+b}$
\end{lem}

\begin{proof}
  We proceed by induction on $a$.
  When $a=0$, the result is trivial.
  Otherwise, let $P=\chain1+(\chain{a-1}\amalg\chain{b+1})$, let $m$ be the maximal element in $\chain{b+1}$, and let $\hat0$ be the unique minimal element of $P$.
  By inductive hypothesis, we know that $\chain{a-1}\amalg\chain{b+1}\embeds\chain{2a+b-1}$, and so $P\embeds\chain{2a+b}$.
  Note that $\D(P,m,\{\hat0\})-m$ is isomorphic to $\chain a\amalg \chain b$, so \[\chain a\amalg\chain b\embeds P \embeds\chain{2a+b},\] proving the result.
\end{proof}

\begin{prop}\label{prop:more_efficient}
  Let $P_\diamond$ denote the poset shown in Figure~\ref{fig:non_poly_poset}. Then $P_\diamond\wkemb\chain{59}$.
\end{prop}

\begin{proof}
  We use the techniques in the proof of Lemma~\ref{lem:max_el} to attach each of the maximal elements to each of the non-maximal elements.
  Most of these relations are already present. For each maximal element, we must only add two relations for each maximal element.
  Next, we dualize and apply the same process to the newly maximal elements (the elements which were minimal in $P_\diamond$).
  For each of these, we must only add three relations per maximal element. The resulting poset, shown in Figure~\ref{fig:p0_smart_embed}.
  Symbolically, this poset can be described as \[P'=(\chain3\amalg\chain3\amalg\chain3)+\anti6+(\chain4\amalg\chain4\amalg\chain4\amalg\chain4).\]

  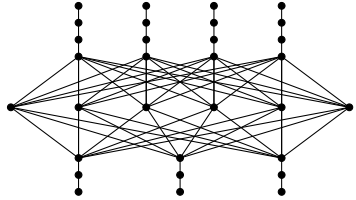
\begin{figure}[h]
    \centering
    \begin{tikzpicture}[scale=0.45]
      \draw[fill] (0,4) circle (0.1cm);
      \draw[fill] (2,4) circle (0.1cm);
      \draw[fill] (4,4) circle (0.1cm);
      \draw[fill] (6,4) circle (0.1cm);
      \draw[fill] (0,4.5) circle (0.1cm);
      \draw[fill] (2,4.5) circle (0.1cm);
      \draw[fill] (4,4.5) circle (0.1cm);
      \draw[fill] (6,4.5) circle (0.1cm);
      \draw[fill] (0,5) circle (0.1cm);
      \draw[fill] (2,5) circle (0.1cm);
      \draw[fill] (4,5) circle (0.1cm);
      \draw[fill] (6,5) circle (0.1cm);
      \draw[fill] (0,5.5) circle (0.1cm);
      \draw[fill] (2,5.5) circle (0.1cm);
      \draw[fill] (4,5.5) circle (0.1cm);
      \draw[fill] (6,5.5) circle (0.1cm);
      \draw[fill] (-2,2.5) circle (0.1cm);
      \draw[fill] (0,2.5) circle (0.1cm);
      \draw[fill] (2,2.5) circle (0.1cm);
      \draw[fill] (4,2.5) circle (0.1cm);
      \draw[fill] (6,2.5) circle (0.1cm);
      \draw[fill] (8,2.5) circle (0.1cm);
      \draw[fill] (0,0) circle (0.1cm);
      \draw[fill] (3,0) circle (0.1cm);
      \draw[fill] (6,0) circle (0.1cm);
      \draw[fill] (0,0.5) circle (0.1cm);
      \draw[fill] (3,0.5) circle (0.1cm);
      \draw[fill] (6,0.5) circle (0.1cm);
      \draw[fill] (0,1) circle (0.1cm);
      \draw[fill] (3,1) circle (0.1cm);
      \draw[fill] (6,1) circle (0.1cm);

      \draw (0,0) -- (0, 5.5);
      \draw (3,0) -- (3, 1);
      \draw (6,0) -- (6, 5.5);
      \draw (2,2.5) -- (2, 5.5);
      \draw (4,2.5) -- (4, 5.5);
      \draw (-2,2.5) -- (0,4);
      \draw (-2,2.5) -- (2,4);
      \draw (-2,2.5) -- (4,4);
      \draw (-2,2.5) -- (6,4);
      \draw (0,2.5) -- (0,4);
      \draw (0,2.5) -- (2,4);
      \draw (0,2.5) -- (4,4);
      \draw (0,2.5) -- (6,4);
      \draw (2,2.5) -- (0,4);
      \draw (2,2.5) -- (2,4);
      \draw (2,2.5) -- (4,4);
      \draw (2,2.5) -- (6,4);
      \draw (4,2.5) -- (0,4);
      \draw (4,2.5) -- (2,4);
      \draw (4,2.5) -- (4,4);
      \draw (4,2.5) -- (6,4);
      \draw (6,2.5) -- (0,4);
      \draw (6,2.5) -- (2,4);
      \draw (6,2.5) -- (4,4);
      \draw (6,2.5) -- (6,4);
      \draw (8,2.5) -- (0,4);
      \draw (8,2.5) -- (2,4);
      \draw (8,2.5) -- (4,4);
      \draw (8,2.5) -- (6,4);

      \draw (-2,2.5) -- (0,1);
      \draw (-2,2.5) -- (3,1);
      \draw (-2,2.5) -- (6,1);
      \draw (0,2.5) -- (0,1);
      \draw (0,2.5) -- (3,1);
      \draw (0,2.5) -- (6,1);
      \draw (2,2.5) -- (0,1);
      \draw (2,2.5) -- (3,1);
      \draw (2,2.5) -- (6,1);
      \draw (4,2.5) -- (0,1);
      \draw (4,2.5) -- (3,1);
      \draw (4,2.5) -- (6,1);
      \draw (6,2.5) -- (0,1);
      \draw (6,2.5) -- (3,1);
      \draw (6,2.5) -- (6,1);
      \draw (8,2.5) -- (0,1);
      \draw (8,2.5) -- (3,1);
      \draw (8,2.5) -- (6,1);

    \end{tikzpicture}
    \caption{A poset $P'$ used as an intermediate step in an embedding sequence for the poset $P_\diamond$ shown in Figure~\ref{fig:non_poly_poset}.}
    \label{fig:p0_smart_embed}
  \end{figure}

  Using Lemma~\ref{lem:two_chains} and dualizing, we obtain that \[P'\wkemb\chain{28}+\anti6+\chain{15}.\]
  Finally, because $\anti6\wkembin{5}\chain{11}$, we know that $\chain{28}+\anti{6}\wkembin{5}\chain{39}$. Applying Lemma~\ref{lem:add_top} yields
  \[P_\diamond\wkemb P'\wkemb\chain{28}+\anti6+\chain{15}\wkembin{10}\chain{59},\] which completes the proof.
\end{proof}

\begin{rmk}\label{rmk:heuristic}
  As a consequence of the preceeding result, one can express $k(U_{59})$
  as a $\zz[q]$-linear combination of terms of the form $k(P)$ and $k(S)$
  for posets $P$ and poset systems $S$ such that one of these terms is~$k(P_\diamond)$.
  It seems implausible that the remaining terms would contribute in such a way
  as to cancel out the contribution of $k(P_\diamond)$, and render $k(U_{59})$ a polynomial,
  thus our Conjecture~\ref{conj:false-59}.
  Unfortunately explicit computation of $k(U_{59})$ is well beyond the capabilities
  of any modern computer, and likely to remain so in the forseeable future.
  (c.f.~\S\ref{ssec:fin-rems-time}).
\end{rmk}

\bigskip

\section{Algorithm and Experimental results}\label{sec:experimental}
\subsection{Algorithm}
Given a poset $P$, to test whether or not $k(P)$ is a polynomial in $q$, we apply the following recursive algorithm.
Pick a maximal element $m$ of $P$ and iterate through all poset systems of the form $S=(P,m,A)$.
If we can apply the equivalences given in lemmas~\ref{lem:remove_max},~\ref{lem:normal_conj}, and~\ref{lem:apply_d} to obtain a reduction of $S$ to some poset $Q$, then recursively compute $k(Q)$.
If there is even one poset system $S$ which cannot be reduced via these methods, we try another maximal element.
If we exhaust all maximal elements in this way, we try the same procedure on the dual poset $P^*$.
If this also fails, we fall back on a slower approach to compute the values $k(S)$ which the algorithm otherwise failed to compute.
This slower approach is a modification of the algorithm discussed in~\cite{VA1,VA2,VA3}.
We call this modification the \dfn{VLA-algorithm}, and give a brief description of the necessary adaptations.

Order the cells in $\low_P$ from bottom to top, and reading each row left to right.
This is the ordering induced from \[(n,1)<(n,2)<\cdots<(n,n-1)<(n-1,1)<\cdots<(3,1)<(3,2)<(2,1).\]
The computation starts at the least cell in the ordering, and iterates through the cells recursively, branching when necessary.
When the algorithm reaches a cell, it attempts to conjugate the cell to zero while fixing all previously seen cells.
If this is possible, the cell is called \dfn{inert}, it sets the cell to zero, and continues on to the next cell in the ordering.
If the cell is not inert, it is called a \dfn{ramification cell}.
The algorithm will branch into two cases: \ts
one where the cell contains a zero, and one where the cell does not.

It often happens that some cells will be inert or ramification cells depending on some algebraic conditions on the previously visited cells.
For example, it may be the case that cell $(5,2)$ will be a ramification cell if $X_{5,1}=X_{6,2}$, and inert otherwise, i.e.~where $X_{i,j}$ denotes the value in cell~$(i,j)$.
In such instances, the algorithm will branch into three different cases:
\begin{enumerate}
  \item the condition to be inert holds, and the cell is set to zero,
  \item the condition to be inert fails (so the cell is a ramification cell), but the cell happens to be zero anyway,
  \item the condition to be inert fails (so the cell is a ramification cell), and the cell is non-zero.
\end{enumerate}
Determining how often the algebraic conditions hold is done with several simple techniques which take care of the vast majority of the algebraic varieties which show up in practice.

To apply the VLA-algorithm to a poset system $S=(P,m,A)$, rather than starting at the beginning of the ordering, we start with some seeded data.
Specifically, we start with the value~$1$ in each cell $(m,a)$, where $a\in A$, and the value~$0$ in each cell $(m,x)$ for $x\not\in A$.

In the two subsequent boxed figures, we provide pseudocode for our algorithm (excluding the VLA-algorithm).
For further details, we refer the reader to our \ts \verb!C++! \ts source code, which is available at \ts
\url{http://www.math.ucla.edu/~asoffer/content/pgcc.zip})\ts.

\begin{figure}[h]
  \begin{flushleft}
    \begin{framed}

      \noindent\textbf{Input:} A poset $P$.\\
      \textbf{Output:} The function $k(P)$.
      \hrule
      \begin{verbatim}

function compute_poset( P ):
  output = 0
  let m in max(P)

  for each A in antichains(P) below m:
    Q = compute_poset_system(P, m, A)
    if Q is not "FAILURE":
      output = output + (q-1)^size(A) * compute_poset(Q)
    else:
      if have not tried some max m':
        restart with m'
      else if have not tried P*:
        compute_poset( P* )
      else:
        output = output + VLA_algorithm(P, m, A)
  return output
      \end{verbatim}
    \end{framed}
  \end{flushleft}
\end{figure}

\begin{figure}[h]
  \begin{flushleft}
    \begin{framed}
      \noindent\textbf{Input:} A poset system $(P, m, A)$.\\
      \textbf{Output:} A poset $Q$ with $k(Q)=k(P,m,A)$, or ``FAILURE'' if none can be found.
      \hrule
      \begin{verbatim}

  function compute_poset_system( P, m, A ):
    while P is changing:
      P = D(P, m, A)
      for a,b in A:
        if above(a) contains above(b) and
           below(b) contains below(a):
          A = A - b

    if no element below m and above member of A:
      return P
    else:
      return "FAILURE"
      \end{verbatim}
    \end{framed}
  \end{flushleft}
\end{figure}

\subsection{Small posets}
Gann and Proctor maintain a list of posets with 9 or fewer elements on their website~\cite{CP}.
We use their lists of connected posets in our verification.
Without using the VLA-algorithm, our code verifies that $k(P)\in\zz[q]$ for every poset $P$ with 7 or fewer elements.
Furthermore, using the VLA-algorithm when necessary as described above, our code verifies that $k(P)\in\zz[q]$ for every poset $P$ with 9 or fewer elements.
Moreover, for each such poset $P$, we have $k(P)\in\nn[q-1]$.
This proves Theorem~\ref{thm:hp-small}.
A text file containing all posets on 9 or fewer elements along with their associated polynomials is available at \url{http://www.math.ucla.edu/~asoffer/kunq/posets.txt}.

\subsection{Chains}
Our code computes $k(U_n)$ for every $n\le 11$ without needing to employ the VLA-algorithm. For $12\le n\le 16$, our code verifies the polynomiality modulo the computation of several ``exceptional poset systems'' which are tackled with the VLA-algorithm.
This verifies the results of Arregi and Vera-L\'opez in~\cite{VA3}, and extends their results to the computation
to all~$n\le 16$.

As $n$ grows, the number of exceptional poset systems which require the use of the VLA-algorithm grows quickly, as shown in Figure~\ref{fig:exceptional_posets} below.\footnote{Computations made with an Intel\textsuperscript{\textregistered}Xeon\textsuperscript{\textregistered} CPU X5650 2.67GHz and 50Gb of RAM.}
The polynomials $k(U_n)$, for $n \le 16$ are given in the Appendix~\ref{sec:app}
and prove Theorem~\ref{thm:n16}.
\begin{figure}[h]
  \centering
  \begin{tabular}{rrrl}
    $n$ & Exceptional poset systems & Computation time (sec.)\\\hline
    $\le11$ & 0 & $\le 0.2$\\
    12 & 1 & 0.5\\
    13 & 8 & 4.4\\
    14 & 64 & 120.7 & ($\sim 2$ minutes)\\
    15 & 485 & 4456 & ($\sim 1.2$ hours)\\
    16 & 3550 & 164557 & ($\sim46$ hours)
  \end{tabular}
  \caption{Computation time and number of exceptional poset
  systems.}
  \label{fig:exceptional_posets}
\end{figure}

\bigskip

\section{Final Remarks}\label{sec:fin-rems}

\subsection{}
Our approach is motivated by the philosophy of Kirillov's orbit method (see~\cite{K1}).
In the case of $U_n(\rr)$, the orbit method provides a correspondence between the
\mbox{irreducible} unitary representations of $U_n(\rr)$ and the co-adjoint orbits.
Moreover, the co-adjoint orbits enjoy the structure of a symplectic manifold.
The unitary characters can actually be recovered by integrating a particular form against
the corresponding orbit.

Over finite fields, a manifold structure is not possible, but some of the philosophy of
the orbit method seems to still be relevant and some formulas translate without
difficulty.  For example, the number of conjugacy classes (and therefore irreducible
repreresentations), is equal to the number of co-adjoint orbits (Lemma~\ref{lem:adcoad}).
However, the naturally analogous character formula does not hold~\cite{IK}.

Note also that the proof of Lemma~\ref{lem:adcoad} makes little use of the
structure of pattern groups. In fact, the theorem holds for any \dfn{algebra group},
defined in~\cite{Isa}.  The proof is in fact an extension of the proof
for~$U_n(\fq)$ defined in~\cite{Isa} (c.f.~\cite{DI}).

\subsection{}
In~\cite{Isa}, Isaacs introduced pattern groups and explained that one can
count characters in $U_n(q)$ by counting characters in stabilizers of a certain
group action (see also~\cite{DT}).
These stabilizers are themselves pattern groups, and lend themselves to a similar recursion,
but over characters, rather than coadjoint orbits.  There is more than superficial difference
between the recursion in~\cite{Isa} and in this paper.  In fact, it follows from~\cite{IK},
that the characters cannot correspond to coadjoint orbits via the natural analogue of
Kirillov's orbit method.

\subsection{}\label{fin_rems_va}
Higman originally stated Conjecture~\ref{conj:higman} in the form of an open problem~\cite{H1};
it received the name ``Higman's Conjecture" more recently.  Higman originally checked
that the conjecture holds for $n\le 5$.  The calculation of the number of conjugacy
classes was later extended to $n\le 8$ by Gudivok et al.\/ in~\cite{G+}
(they made a mistake for $n=9$). The authors use a variation on the brute force algorithm.

Later, Arregi and Vera-L\'opez verified Higman's conjecture for $n\le 13$ in~\cite{VA3}
by a clever application of a brute force algorithm for counting adjoint orbits.
They also proved that the number of conjugacy classes of cardinality $q^s$ is polynomial
for $s\le n-3$ \cite{VA2}.  Moreover, they verified that, as a polynomial in $(q-1)$,
the number of conjugacy classes of cardinality $q^s$ has non-negative integral coefficients
(for $s\le n-3$).  For other partial results Higman's conjecture see also~\cite{ABT,Isa,Mar}.
We refer to~\cite{Sof-thesis} for a broad survey of the literature on the conjugacy classes
of $U_n(q)$, both the algebraic and combinatorial aspects.

\subsection{}
In recent years, much effort has been made to improve Higman's upper bounds
for the asymptotics of $k(U_n(q))$, as $n\to \infty$, see~\cite{Mar,Sof,VA0}.
For a fixed~$q$, it is conjectured that
\begin{equation}
  k(U_n(q)) \, = \, q^{\frac{n^2}{12}\ts (1+o(1))} \quad \text{as} \ \. n\to \infty\ts.\tag{$\lozenge$}
\end{equation}
The lower bound is known and due to Higman in the original paper~\cite{H1}, while the
best upper bound is due to the second author~\cite{Sof}~:

\[q^{\frac{n^2}{12}\ts (1+o(1))} \le k(U_n(q)) \, \le \,
q^{\frac{7}{44}\ts n^2\ts (1+o(1))} \quad \text{as} \ \. n\to \infty\ts.\]

The above asymptotics have curious connection to this work.  Arregi and Vera-L{\'o}pez
conjectured in~\cite{VA3} a refinement of Higman's Conjecture~\ref{conj:higman}
stating that the degree of the polynomials $k(U_n)$ are
equal to  $\floor{n(n+6)/12}$.  
If true, this would confirm the asymptotics~$(\lozenge)$ as well.
While we do not believe Higman's conjecture, the degree formula continues to hold for new values,
so it is now known for all $n\le 16$.

\subsection{}\label{ssec:fin-rems-posets}
In~\cite{HP}, Halasi and P\'alfy exhibit a pattern group for which the number of conjugacy
classes is not a polynomial in the size of the field.
Though they do not provide explicit bounds, their construction yields a \ts 5,592,412-element poset.
We obtained the 13-element poset $P_\diamond$ shown in Figure~\ref{fig:non_poly_poset} by modifying
their construction.

It would be interesting to see if the poset~$P_0$ is in fact the smallest posets with a non-polynomial
$k\bigl(U_{P_0}\bigr)$.  By Theorem~\ref{thm:hp-small}, such posets would have to have at least~10 elements.
Unfortunately, even this computation might be difficult since the total number of connected posets is
rather large. For example, there are about $1.06 \cdot 10^9$ connected posets on 12 elements,
see e.g.~\cite{BM} and \cite[A000608]{OEIS}.

\subsection{}\label{ssec:fin-rems-parallel}
When our algorithm falls back on the VLA-algorithm, the poset systems it must compute
have minimal shared computational resources.  For this reason, our technique lends itself
well towards parallelization.  This, along with several optimization techniques we believe
could be used to compute $k(U_{17}(q))$ and $k(U_{18}(q))$.  However, due to the super-exponential
growth rate of $k(U_n(q))$, pushing the computation significantly further will likely
require different techniques.

\subsection{}\label{ssec:fin-rems-time}
Based on our computations, one can try to give a conservative lower bound to the cost of computing $k(U_{59}(q))$. 
Assuming the current rate of increase in timing, we estimate our algorithm to need about $10^{66}$ years of CPU time. 
Alternatively, if we assume Moore's~law\footnote{Moore's law is the observation that the number of transistors per square inch on an integrated circuit has been doubling roughly every 18 months. Quite roughly, this can be interpretted as computer performance increase.} will continue to hold indefinitely, this computation will not become feasible until the year~2343.

\subsection{}\label{ssec:fin-rems-59}
There are two directions in which the bound $n\ge 59$ in Conjecture~\ref{conj:false-59} can be
decreased.  First, it is perhaps possible that $P_\diamond$ embeds into a smaller chain.  This is
a purely combinatorial problem which perhaps also lends to computational solution.
We would be interested to see if such improvement is possible.

Second, it is conceivable and perhaps likely that there are posets $P$ with
more than~13 elements which embed into $C_n$ with $n<59$, and have non-polynomial $k(U_P)$.
Since $P_\diamond$ really encodes the variety $x^2=1$, it would be natural to consider other
algebraic varieties which have different point counts depending on teh characteristic.
This is a large project which goes beyond the scope of this work.

\subsection{} \label{ssec:fin-kirillov}
In~\cite{K2,K3}, Kirillov made two conjectures on the values of
$k(U_n(q))$ for small~$q$.

\begin{conj}[Kirillov] \label{conj:kirillov-seq}
For all $n\ge 1$, we have \ts $k\bigl(U_n(2)\bigr) \ge \ra_{n+1}$ \ts and \ts $k\bigl(U_n(3)\bigr) \ge \rb_{n+1}$,
where $\{\ra_n\}$ is the Euler sequence and $\{\rb_n\}$ is the Springer sequence.
\end{conj}

Here the \emph{Euler sequence} $\{\ra_n\}$ counts the number of \emph{alternating
permutations} $\si\in S_n$; it has elegant generating function and asymptotics:
\[\sum_{n=0}^\infty \. \ra_n \. \frac{x^n}{n!} \, = \, \sec(x) + \tan(x)\., \qquad
\ra_n \. \sim \. \frac{4}{\pi} \. \left(\frac{2}{\pi}\right)^n n!,\]
see~\cite[A000111]{OEIS}.  Similarly, the \emph{Springer sequence} $\{\rb_n\}$
counts the number of \emph{alternating signed permutations} in the hyperoctahedral group $C_n$;
it has elegant generating function and asymptotics:
\[\sum_{n=0}^\infty \. \rb_n \. \frac{x^n}{n!} \, = \, \frac{1}{\cos(x) - \sin(x)}\., \qquad
\rb_n \. \sim \. \frac{2\sqrt{2}}{\pi} \. \left(\frac{4}{\pi}\right)^n n!,\]
see~\cite[A001586]{OEIS}.

Kirillov observed that there is a remarkable connection between the sequences
(see Appendix~\ref{sec:app-kirillov}), and made further conjectures related to them.
It is easy to see that asymptotics imply the conjecture for large~$n$.  By using
exact values of $\{\ra_n\}$ and $\{\rb_n\}$, and technical improvements on the
lower bounds by Higman, is easy to show the bounds in the conjecture hold for
$n\ge 43$ and $n\ge 30$, respectively~\cite{Sof-thesis}.  Our results confirm the
conjecture for $n\le 16$, leaving it open only for the intermediate values in both cases.

\subsection{}\label{ssec:fin-rems-alperin}
Recall that by the Halasi--P\'alfy theorem, the functions $k(\UP(q))$ be as bad as any algebraic variety~\cite{HP}.
Theorem~\ref{thm:chain_univ} suggests that $k(U_n(q))$ is also this bad.
This would be in line with other universality results in algebra and geometry,
see e.g.~\cite{BB,Mnev,Vak}.

In a different direction, Alperin showed that the action of $U_n$ by conjugation
on $\GL_n$ does have polynomial behavior~\cite{Alp}.
Specifically, he showed \[\abs{\GL_n/U_n}\in\zz[q]\] for all $n>0$.
Moreover, because $U_n$ acts by conjugation on each cell of the
Bruhat decomposition of $\GL_n$, we have
\[
\abs{\GL_n/U_n}\, = \, \sum_{w\in S_n}\. \abs{B_nwB_n/U_n}\ts.
\]
The term in the summation corresponding to the identity element of $S_n$ is $\abs{B_n/U_n}$,
which bears resemblance to $k(U_n)$.
Complementary to our heuristic in Remark~\ref{rmk:heuristic},
Alperin noted that it seems unlikely that the summation on the right-hand
side has even one non-polynomial term, given that the left-hand side is a polynomial.

For another similar phenomenon, let us mention that there are many moduli spaces
which satisfy \emph{Murphy's law}, a version of Mn\"ev's Universality Theorem~\cite{Vak}.
Over~$\fq$, these moduli spaces have a non-polynomial number of points.
But of course, when summed over all possible configurations these functions
of~$q$ add up to a polynomial, the size of the Grassmannian or other flag varieties.

To reconcile these examples with our main approach, think of them as different examples
of counting points on orbifolds.  Apparently, both the Grassmannian and Alperin's actions
are \emph{nice}, while conjugation on $U_n(\fq)$ is not.  This is not very surprising.
For example, both binomial coefficients~$\binom{n}{k}$ and the number of integer
partitions $p(n)$ count the orbits of certain combinatorial actions
(see the \emph{twelvefold way}~\cite{Sta}).  However, while the former
are ``nice'' indeed, the latter are notoriously complicated.  Despite a large body
of work on partitions, from Euler to modern times, little is known about
divisibility of~$p(n)$; for example, the \emph{Erd\H{o}s conjecture} that
every prime~$s$ is a divisor of some $p(n)$ remains wide open
(see e.g.~\cite{AO}).\footnote{Naturally, one would assume that
asymptotically, we have \ts $s\ts|\ts p(n)$ \ts for a positive
fraction of~$n$.  This is known for some primes~$s$, such as $5, 7$
and~$11$ due to \emph{Ramanujan's congruences}, but is open for $2$
and~$3$, see e.g.~\cite{AO}. }
This suggests that certain numbers of orbits are so wild, that
even proving that they are wild is a great challenge.

\subsection{}
There is little hope of finding interesting classes of posets for which $k(U_P)$ is always a polynomial.
If such a class $\cal P$ contains posets of arbitrary height, then one can show that all chains embed in some member $\cal P$.
As embedding is a transitive property, the family $\cal P$ shares the same universality properties that $\{\chain n\}$ has.
Even the posets of height no more than three can be as bad as arbitrary algebraic varieties \cite{HP}.
Of course, if $P$ is a poset of height two, then $U_P$ is abelian and therefore $k(U_P)$ is a polynomial in $q$, however this family is not interesting.

\subsection{}
Pattern groups and pattern algebras are closely related to incidence algebras of posets.
In fact, the incidence algebra $I(P)$ of a poset $P$ contain, up to isomorphism,
both $\UP$ and $\U_P$ as substructures. This perhaps provides a more natural setting
for the group $\UP$, as it is independent of any total ordering we assign to the elements
of~$P$.  For more information on incidence algebras, see~\cite{Rota,Sta}.

\subsection{}
Many group-theoretic constructions have combinatorial interperetations when
applied to pattern groups.  For instance, the intersection of two pattern groups
$\UP$ and $U_Q$ is the pattern group defined on the poset $P\cap Q$.  For a less
trivial example, the commutator subgroup of $\UP$ is a pattern group $U_{\hspace{-0.6mm}P'}$,
where $P'$ is the subposet consisting of the non-covering relations in $P$.
Normalizers (in $U_n$) are also pattern groups, and can be expressed combinatorially.
However, for normalizers, the precise injection of $U_P$ into $U_n$ is relevant.
As such, the normalizer depends not only on the poset $P$ but also its linear extension.

\vskip.4cm

\noindent
{\bf Acknowledgements.}  We are very grateful to a number of people
for many interesting conversations and helpful remarks: \ts
Karim Adiprasito, Persi Diaconis, Scott Garrabrant, Robert Guralnick, Martin Isaacs,
Alexandre Kirillov, Eric Marberg, Brendan McKay, Alejandro Morales, Peter M. Neumann,
Greta Panova, Rapha\"el Rouquier, and Antonio Vera-L\'opez.
The first author was partially supported by the~NSF.

\vskip.7cm


\appendix
\section{Polynomials $k\bigl(U_n(q)\bigr)$, $q=t+1$} \label{sec:app}
{\footnotesize
\begin{align*}
  k(U_{1}) =\ & 1\\
  k(U_{2}) =\ & 1 + \ts t\\
  k(U_{3}) =\ & 1 + 3\ts t + \ts t^2\\
  k(U_{4}) =\ & 1 + 6\ts t + 7\ts t^2 + 2\ts t^3\\
  k(U_{5}) =\ & 1 + 10\ts t + 25\ts t^2 + 20\ts t^3 + 5\ts t^4 \\
  k(U_{6}) =\ & 1 + 15\ts t + 65\ts t^2 + 105\ts t^3 + 70\ts t^4 + 18\ts t^5 + \ts t^6\\
  k(U_{7}) =\ & 1 + 21\ts t + 140\ts t^2 + 385\ts t^3 + 490\ts t^4 + 301\ts t^5 + 84\ts t^6 + 8\ts t^7\\
  k(U_{8}) =\ & 1 + 28\ts t + 266\ts t^2 + 1120\ts t^3 + 2345\ts t^4 + 2604\ts t^5 + 1568\ts t^6 + 496\ts t^7 + 74\ts t^8 + 4\ts t^9\\
  k(U_{9}) =\ & 1 + 36\ts t + 462\ts t^2 + 2772\ts t^3 + 8715\ts t^4 + 15372\ts t^5 + 15862\ts t^6 + 9720\ts t^7 + 3489\ts t^8\\
  & + 701\ts t^9 + 72\ts t^{10} + 3\ts t^{11}\\
  k(U_{10}) =\ & 1 + 45\ts t + 750\ts t^2 + 6090\ts t^3 + 26985\ts t^4 + 69825\ts t^5 + 110530\ts t^6 + 110280\ts t^7\\
  & + 70320\ts t^8 + 28640\ts t^9 + 7362\ts t^{10} + 1170\ts t^{11} + 110\ts t^{12} + 5\ts t^{13}\\
  k(U_{11}) =\ & 1 + 55\ts t + 1155\ts t^2 + 12210\ts t^3 + 72765\ts t^4 + 261261\ts t^5 + 592207\ts t^6 + 877030\ts t^7\\
  & + 868725\ts t^8 + 583550\ts t^9 + 267542\ts t^{10} + 83909\ts t^{11} + 18007\ts t^{12} + 2618\ts t^{13}\\
  & + 242\ts t^{14} + 11\ts t^{15}\\
  k(U_{12}) =\ & 1 + 66\ts t + 1705\ts t^2 + 22770\ts t^3 + 176055\ts t^4 + 841302\ts t^5 + 2600983\ts t^6 + 5387646\ts t^7\\
  & + 7680310\ts t^8 + 7684820\ts t^9 + 5473050\ts t^{10} + 2803182\ts t^{11} + 1042181\ts t^{12} + 284109\ts t^{13}\\
  & + 57256\ts t^{14} + 8484\ts t^{15} + 890\ts t^{16} + 60\ts t^{17} + 2\ts t^{18}\\
  k(U_{13}) =\ & 1 + 78\ts t + 2431\ts t^2 + 40040\ts t^3 + 390390\ts t^4 + 2403258\ts t^5 + 9766471\ts t^6 + 27116232\ts t^7 \\
  & + 52873678\ts t^8 + 74012653\ts t^9 + 75670881\ts t^{10} + 57294120\ts t^{11} + 32515314\ts t^{12}\\
  & + 14000495\ts t^{13} + 4635125\ts t^{14} + 1195116\ts t^{15} + 241436\ts t^{16} + 37778\ts t^{17} + 4381\ts t^{18} \\
  & + 338\ts t^{19} + 13\ts t^{20}\\
  k(U_{14}) =\ & 1 + 91\ts t + 3367\ts t^2 + 67067\ts t^3 + 805805\ts t^4 + 6225219\ts t^5 + 32296264\ts t^6 + 116332645\ts t^7\\
  & +  298956658\ts t^8 + 560602042\ts t^9 + 781499719\ts t^{10} + 822549728\ts t^{11} + 662497381\ts t^{12}\\
  & + 413509705\ts t^{13} + 202666910\ts t^{14} + 79124292\ts t^{15} + 24968979\ts t^{16} + 6441876\ts t^{17}\\
  &  + 1362732\ts t^{18} + 233758\ts t^{19} + 31542\ts t^{20} + 3159\ts t^{21} + 210\ts t^{22} + 7\ts t^{23}\\
  k(U_{15}) =\ & 1 + 105\ts t + 4550\ts t^2 + 107835\ts t^3 + 1566565\ts t^4 + 14864850\ts t^5 + 96136040\ts t^6 + 437680815\ts t^7\\
  & + 1440259535\ts t^8 + 3502779995\ts t^9 + 6416611201\ts t^{10} + 8998108665\ts t^{11} + 9796436195\ts t^{12}\\
  & + 8387410675\ts t^{13} + 5718426690\ts t^{14} + 3145744973\ts t^{15} + 1416179446\ts t^{16} + 529371274\ts t^{17}\\
  & + 166405370\ts t^{18} + 44325415\ts t^{19} + 9997955\ts t^{20} + 1887955\ts t^{21} + 291345\ts t^{22} + 35270\ts t^{23}\\
  & + 3130\ts t^{24} + 180\ts t^{25} + 5\ts t^{26}\\
  k(U_{16}) =\ & 1 + 120\ts t + 6020\ts t^2 + 167440\ts t^3 + 2894710\ts t^4 + 33137104\ts t^5 + 261929668\ts t^6\\
  & + 1475199440\ts t^7 + 6072906125\ts t^8 + 18674026800\ts t^9 + 43703418616\ts t^{10}\\
  & + 79124540872\ts t^{11} + 112420822696\ts t^{12} + 126975887444\ts t^{13} + 115398765556\ts t^{14}\\
  & + 85415064915\ts t^{15} + 52146190588\ts t^{16} + 26615252562\ts t^{17} + 11515549082\ts t^{18}\\
  & + 4278222573\ts t^{19} + 1378103758\ts t^{20} + 386616800\ts t^{21} + 94259304\ts t^{22} + 19784488\ts t^{23}\\
  & + 3513854\ts t^{24} + 514128\ts t^{25} + 59504\ts t^{26} + 5104\ts t^{27} + 288\ts t^{28} + 8\ts t^{29}
\end{align*}
}

{\normalsize

\section{Known values for the Kirilov sequences} \label{sec:app-kirillov}

\noindent
Here we present a table with Kirillov sequences 
$\bigl\{k(U_{n}(2))\bigr\}$ and $\bigl\{k(U_{n}(3))\bigr\}$ 
discussed in~$\S$\ref{ssec:fin-kirillov}. New values computed 
in this paper are shown in red.  Note that the sequences coincide
in the beginning and the ratios increase to 
$k(U_{16}(2))/\ra_{17} \approx 3.2$,
and $k(U_{16}(3))/\rb_{16} \approx 23.6$, suggesting 
that both parts of Kirillov's Conjecture~\ref{conj:kirillov-seq}
are likely to be true.  

\vskip.5cm

\

\begin{tabular}{lr}
\begin{tabular}{c|r|r}
  $n$ & $\ra_{n+1}$ & $k(U_{n}(2))$\\\hline
  1 & 1 & 1\\
  2 & 2 & 2\\
  3 & 5 & 5\\
  4 & 16 & 16\\
  5 & 61 & 61\\
  6 & 272 & 275\\
  7 & 1385 & 1430\\
  8 & 7936 & 8506 \\
  9 & 50521 & 57205\\
  10 & 353792 & 432113\\
  11 & 2702765 & 3641288\\
  12 & 22368256& 34064872\\
  13 & 199360981 & 352200229\\
  14 & 1903757312 & {\color{red} 4010179157}\\
  15 & 19391512145 & {\color{red} 50124636035}\\
  16 & 209865342976 & {\color{red} 685996839568}
\end{tabular}
& \qquad 
\begin{tabular}{c|r|r}
  $n$ & $\rb_n$ & $k(U_n(3))$\\\hline
  1 & 1 & 1\\
  2 & 3 & 3\\
  3 & 11 & 11\\
  4 & 57 & 57\\
  5 & 361 & 361\\
  6 & 2763 & 2891\\
  7 & 24611 & 27555\\
  8 & 250737 & 315761\\
  9 & 2873041 & 4246737\\
  10 & 36581523 & 66999699\\
  11 & 512343611 & 1226296635\\
  12 & 7828053417 & 26011112361\\
  13 & 129570724921 & 635526804025\\
14 & 2309644635483 & {\color{red} 17881012846299}\\
  15 & 44110959165011 & {\color{red} 577907517043923}\\
  16 & 898621108880097 & {\color{red} 21474199259637473}
\end{tabular}
\end{tabular}

}
\end{document}